\documentclass[12pt]{article}

\usepackage{fullpage,amssymb,amsthm,amsfonts,enumerate,textcomp, eurosym, titling, epsfig,epstopdf, authblk,pinlabel,xcolor,caption, graphicx, tikz-cd, xcolor, mathtools}
\usepackage{subcaption}
\usetikzlibrary{shapes}

\usepackage{amsmath}
\makeatletter
\def\grd@save@target#1{%
  \def\grd@target{#1}}
\def\grd@save@start#1{%
  \def\grd@start{#1}}
  
\newcommand{\KB}[1]{%
  \begin{tikzpicture}[baseline=-\dimexpr\fontdimen22\textfont2\relax]
  #1
  \end{tikzpicture}%
}
\newcommand{\KBCirc}{%
  \KB{\filldraw[color=gray, fill=none, thick] circle (0.3);}%
}
\newcommand{\KBCross}{%
  \KB{
    \draw[color=gray,thick] (-0.3,0.3) -- (0.3,-0.3);
    \draw[color=gray,thick] (-0.3,-0.3) -- (-0.05,-0.05);
    \draw[color=gray,thick] (0.05,0.05) -- (0.3,0.3);
  }%
}
\newcommand{\KBH}{%
  \KB{%
    \draw[color=gray,thick] (-0.3,0.3) .. controls (0,-0.05) .. (0.3,0.3);
    \draw[color=gray,thick] (-0.3,-0.3) .. controls (0,0.05) .. (0.3,-0.3);
  }%
}
\newcommand{\KBV}{%
  \KB{%
    \draw[color=gray,thick] (-0.3,-0.3) .. controls (0.05,0) .. (-0.3,0.3);
    \draw[color=gray,thick] (0.3,-0.3) .. controls (-0.05,0) .. (0.3,0.3);
  }%
}

\begin{document}

\makeatother

\title{Asymptotic additivity of the Turaev--Viro invariants for a family of $3$-manifolds}
\author{Sanjay Kumar and Joseph M. Melby}
\date{\vspace{-5ex}}

\newcommand{\Addresses}{{% additional braces for segregating \footnotesize
  \bigskip
  \footnotesize

  \textsc{Department of Mathematics, University of California, Santa Barbara, Santa Barbara,  CA, 93106-6105, USA}\par\nopagebreak
  \textit{E-mail address}: \texttt{sanjay\_kumar@ucsb.edu}\\
 
  \textsc{Department of Mathematics, Michigan State University,
    East Lansing, MI, 48824, USA}\par\nopagebreak
  \textit{E-mail address}: \texttt{melbyjos@msu.edu}

}}

\maketitle

\begin{abstract}
In this paper, we show that the Turaev--Viro invariant volume conjecture posed by Chen and Yang  is preserved under gluings of toroidal boundary components for a  family of $3$-manifolds. In particular, we show that the asymptotics of the Turaev--Viro invariants are additive under certain gluings of elementary pieces arising from a construction of hyperbolic cusped $3$-manifolds due to Agol. The gluings of the elementary pieces are known to be additive with respect to the simplicial volume. This allows us to construct families of manifolds which have an arbitrary number of hyperbolic pieces and satisfy an extended version of the Turaev--Viro invariant volume conjecture.\\

\end{abstract}

\newtheorem{innercustomgeneric}{\customgenericname}
\providecommand{\customgenericname}{}
\newcommand{\newcustomtheorem}[2]{%
  \newenvironment{#1}[1]
  {%
   \renewcommand\customgenericname{#2}%
   \renewcommand\theinnercustomgeneric{##1}%
   \innercustomgeneric
  }
  {\endinnercustomgeneric}
}

\newcustomtheorem{customthm}{Theorem}
\newcustomtheorem{customlemma}{Lemma}
\newcustomtheorem{customprop}{Proposition}
\newcustomtheorem{customconjecture}{Conjecture}
\newcustomtheorem{customcor}{Corollary}

\theoremstyle{plain}
\newtheorem*{ack*}{Acknowledgements}
\newtheorem{thm}{Theorem}[section]
\newtheorem{lem}[thm]{Lemma}
\newtheorem{prop}[thm]{Proposition}
\newtheorem{cor}[thm]{Corollary}
\newtheorem{predefinition}[thm]{Definition}
\newtheorem{conjecture}[thm]{Conjecture}
\newtheorem{preremark}[thm]{Remark}
\newenvironment{remark}%
  {\begin{preremark}\upshape}{\end{preremark}}
 \newenvironment{definition}%
  {\begin{predefinition}\upshape}{\end{predefinition}}

\newtheorem{ex}[thm]{Example}

%\renewcommand\thesubsection{\thesection\Alph{subsection}}

%%%%%%%%%%%%%%%%%%%%%%%%%%%%%%%%%%%%%%%%%%%%%%%%%%%%%%%%%%
%%% SECTION 1
%%%%%%%%%%%%%%%%%%%%%%%%%%%%%%%%%%%%%%%%%%%%%%%%%%%%%%%%%%

\section{Introduction}\label{IntroSec}
Originally constructed by Turaev and Viro \cite{TuraevViro}, the  Turaev--Viro invariants $TV_r(M;q)$ for a compact $3$-manifold $M$ are a family of invariants parameterized by an integer $r\geq 3$ dependent on a $2r$-th root of unity $q$. For this paper, we will be concerned with the $SU(2)$-version of the Turaev--Viro invariants; however, the results hold for the $SO(3)$-version with minor changes. 

In defining these two versions of the Turaev--Viro invariants, the distinction arises from the construction of the topological quantum fields theories of the Reshetikhin--Turaev invariants by Blanchet, Habegger, Masbaum, and Vogel \cite{BHMVKauffman}. In the authors' work, the elements of an index set determined by $r$ correspond to the irreducible representations of $SU(2)$. As $SU(2)$ is a double-covering of $SO(3)$, the authors remark that the $SO(3)$ theory can be obtained as the restriction of the elements of the index set to elements with corresponding representations that lift to $SO(3)$. In addition to considering  different roots of unity, this is realized as requiring that $r$ is odd for $SO(3)$ as opposed to any integer $r$ for $SU(2)$. By following the construction, the Turaev--Viro invariants can also be defined to have an $SU(2)$-version and an $SO(3)$-version. For more details between these two versions of the Turaev--Viro invariants, we refer to Sections $2$ and $3$ of \cite{colJvolDKY} by  Detcherry,  Kalfagianni,  and  Yang where the $SO(3)$ Turaev--Viro invariants are defined.

In regards to the $SU(2)$-version, we focus on a conjecture stated by Chen and Yang in \cite{chen2018volume} which relates the growth rate of the Turaev--Viro invariants for any hyperbolic manifold to this manifold’s hyperbolic volume. Also in \cite{chen2018volume}, the authors provide computational evidence supporting the conjecture.  The conjecture is given as follows. 

\begin{conjecture}[\cite{chen2018volume}, Conjecture 1.1]\label{CYvol}
 Let $TV_r(M;q)$ be the Turaev--Viro invariant of a hyperbolic $3$-manifold $M$, and let $vol(M)$ be the hyperbolic volume of $M$. For $r$ running over odd integers and $q = e^{\frac{2\pi \sqrt{-1}}{r}}$, 
\[
\lim_{r \rightarrow \infty} \frac{2\pi}{r} \log |TV_r(M;q) | = vol(M).
\]
\end{conjecture}

As a natural extension to manifolds which are not hyperbolic, Conjecture \ref{CYvol} has been restated by Detcherry and Kalfagianni \cite{detcherry2019gromov} in terms of the simplicial volume. (See Conjecture \ref{GNconj} below.)  For more details on the simplicial volume, we refer to \cite{Gromov} by Gromov and \cite{ThurstonGT3manifolds} by Thurston. For our purposes, we will only need the relationship between the hyperbolic volume and the simplicial volume of a hyperbolic $3$-manifold $M$ given by 
$$vol(M) = v_3 \| M \|$$
where $v_3 \approx 1.0149$ is the volume of a regular ideal tetrahedron and $\|\cdot\|$ is the simplicial volume.

\begin{conjecture}[\cite{detcherry2019gromov}, Conjecture 8.1]\label{GNconj} 
Let $M$ be a  compact and orientable $3$-manifold with empty or toroidal boundary. For $r$ running over odd integers and $q = e^{\frac{2\pi \sqrt{-1}}{r}}$,
\[
\limsup_{r \rightarrow \infty} \frac{2\pi}{r} \log |\text{TV}_r(M;q)| = v_3 ||M||.
\]
\end{conjecture}

\begin{remark}
Conjectures \ref{CYvol} and \ref{GNconj} should be compared to the well-known volume conjecture of Kashaev \cite{KashaevVolConj} which similarly relates the growth rate of the Kashaev invariant to the hyperbolic volume of hyperbolic link complements. By a result from Murakami and Murakami \cite{MurakamiMurakamiVolConj}, the Kashaev volume conjecture is more commonly written in terms of the colored Jones polynomials. 
\end{remark}

 An important property of any compact irreducible orientable $3$-manifold is that it can be cut along a unique (up to isotopy) minimal collection of incompressible tori into atoroidal $3$-manifolds through the JSJ decomposition \cite{JacoShalen, Johannson}. Furthermore, these atoroidal manifolds are each either hyperbolic or Seifert fibered by Thurston's Geometrization Conjecture \cite{ThurstonGeomConj} which was famously completed following the work of  Perelman \cite{PerelmanEntropy, PerelmanFinite, PerelmanRicci}. For more details on the Geometrization Conjecture, we refer to Chapter 12 of Martelli's book \cite{Martelli}. For a given manifold $M$, its simplicial volume is equal to the sum of the simplicial volumes of the pieces of its JSJ decomposition. The simplicial volume is positive for each hyperbolic piece and zero for each Seifert fibered piece. We expect the same additivity relationship to hold asymptotically for the Turaev--Viro invariants as reflected in Conjecture \ref{GNconj}.

The \emph{asymptotic additivity} property of the Turaev--Viro invariants has previously been shown for a few families of manifolds. For a manifold $M$ which satisfies Conjecture \ref{GNconj}, the property was proven for so-called invertible cablings of $M$ by Detcherry and Kalfagianni \cite{detcherry2019gromov} and $(p,2)$-torus knot cablings of $M$ by Detcherry \cite{DetcherryCabling}. Each of these results involve gluing a Seifert-fibered manifold to another manifold, which does not change simplicial volume. Additionally, it was proven for the figure-eight knot cabled with Whitehead chains by Wong \cite{wong2020WHFig8}, providing examples of gluing additivity for certain pairs of hyperbolic manifolds. Our construction is the first which glues several hyperbolic pieces to produce infinite families of manifolds satisfying the asymptotic additivity property.

As our main result, the following Theorem \ref{mainthm} establishes the asymptotic additivity property for an infinite family of manifolds glued from several hyperbolic pieces.
Our construction is inspired by a construction of Agol \cite{AgolSmall} of cusped $3$-manifolds with well-understood geometric properties. Agol begins with an oriented $S^1$-bundle over a surface and systematically drills out curves to produce octahedral link complements. This procedure depends on a path on the $1$-skeleton of the pants complex of the surface. The hyperbolic building blocks for our family $\mathcal{M}$ of manifolds are obtained as follows: We begin with a trivial $S^1$-bundle over the once-punctured torus and use Agol's procedure to drill out a $2$-component link. This produces a hyperbolic manifold, which we call an $S-$piece, of volume $2v_8$, where $v_8 \approx 3.66$ is the volume of the regular ideal hyperbolic octahedron. Then we begin with a trivial $S^1$-bundle over the four-punctured sphere and use Agol's procedure to drill out a $2$-component link, producing a hyperbolic manifold of volume $4v_8$ which we call an $A-$piece. Gluing $k$ $S-$pieces and $l$ $A-$pieces along their original boundaries produces a compact manifold $M_L(k,l) \in \mathcal{M}$, where $L$ is the union of the link components of the $S-$ and $A-$pieces. For more details, see Subsection \ref{linkfamilySubSec}.

 \begin{customthm}{\ref{mainthmrestate} }\label{mainthm}
Let $M_L(k,l) \in \mathcal{M}$. Then for $r$ running over odd integers and $q = e^{\frac{2\pi \sqrt{-1}}{r}}$,
\begin{align*}
\lim_{r \rightarrow \infty} \frac{2\pi}{r} \log |TV_r(M_L(k,l);q)| = v_3 ||M_L(k,l)|| = 2(k+2l)v_8.
\end{align*} 
\end{customthm}

In general, the asymptotic additivity property is difficult to prove. In order to simplify the calculation, the family $\mathcal{M}$ was constructed to have several advantageous properties which we utilize. We first remark that the Turaev--Viro invariants can be computed from the relative Reshetikhin--Turaev invariants by a result of Belletti, Detcherry, Kalfagianni, and Yang in \cite{growth6j}.  Our family $\mathcal{M}$ can be described effectively from Turaev's shadow perspective \cite[Section 3]{TuraevShadow} which Turaev related to the relative Reshetikhin--Turaev invariants in \cite[Chapter X]{TuraevBook}. Additionally, manifolds in the family $\mathcal{M}$ have relative Reshetikhin--Turaev invariants which are comparably simple to manage as well as well-understood simplicial volumes. 

We note that the consideration of the  shadow perspective was taken from the following works. In \cite{costantinoColoredJones},  Costantino extended the colored Jones invariants to links in $S^3 \#_k \left(S^2 \times S^1\right)$ and used the formulation of the invariant to prove a version of the volume conjecture for a family of links in $S^3 \#_k \left(S^2 \times S^1\right)$ known as the \emph{fundamental shadow links}. Furthermore in \cite{growth6j},  Belletti, Detcherry, Kalfagianni, and Yang  represented the Turaev--Viro invariants in terms of the relative Reshetikhin--Turaev invariants, which they used in combination with Costantino's formulation \cite{Costantino6j2007} to show the fundamental shadow links satisfy Conjecture \ref{GNconj}. In \cite{WongYang}, Wong and Yang also use this shadow viewpoint to study a version of the volume conjecture involving the relative Reshetikhin--Turaev invariants. In this paper, we utilize the same approach to prove Theorem \ref{mainthm}; however, we note that the form of the relative Reshetikhin--Turaev invariants  we study here is more complicated.

Moreover, the Turaev--Viro invariants are related to a measure of complexity of a manifold called the \textit{shadow complexity} derived from Turaev's shadow perspective for $3$-manifolds. We refer to Costantino and Thurston \cite{CostantinoThurston} or Turaev \cite{TuraevBook} for more details. The shadow complexity $c \in \mathbb{N}$ of a manifold gives a sharp upper bound for the growth rate of its Turaev--Viro invariants as stated in the following.

\begin{cor}[\cite{growth6j}, Corollary 3.11]
If $M$ has shadow complexity $c$, then
\begin{align*}
    lTV(M) \leq LTV(M) \leq 2c v_8,
\end{align*}
where $\displaystyle lTV(M) = \liminf_{r \rightarrow \infty} \frac{2\pi}{r} \log |TV_r(M;q)|$ and $\displaystyle LTV(M) = \limsup_{r \rightarrow \infty} \frac{2\pi}{r} \log |TV_r(M;q)|$. Furthermore, we have equalities for fundamental shadow links.
\end{cor}

In a similar way, the manifolds $M_L(k,l)$ have a shadow complexity based on the elementary pieces used in their construction such that they satisfy the same equalities as the fundamental shadow links as shown in Theorem \ref{mainthm}. In terms of  the shadow construction  described in Subsection \ref{Shadowsubsec},  the shadow complexity is the number of the shadow's vertices  $c=k+2l$.

The paper is organized as follows: We recall Agol's construction of cusped $3$-manifolds and introduce the family of manifolds $\mathcal{M}$ in Section \ref{AgolandLinksSec}. In Section \ref{TVfromShadowSec}, we introduce Turaev's shadow invariant and discuss its relationship with the relative Reshetikhin--Turaev and Turaev--Viro invariants. We include definitions of the relative Reshetikhin--Turaev and Turaev--Viro invariants in Section \ref{TVfromShadowSec} in order to provide context for the interplay between the invariants, but we stress that these precise definitions are not needed to understand the proof of Theorem \ref{mainthm}. This proof comprises Section \ref{mainThmProofSec}, and lastly, Section \ref{FurtherDirections} consists of further directions for this project.

\begin{ack*}
The authors would like to thank their advisor Efstratia Kalfagianni for guidance and helpful discussions. In addition, the authors would like to thank Renaud Detcherry and Tian Yang for their comments and suggestions on this paper. Finally, the authors would like to thank the referee for their time and careful reading, as well as providing valuable suggestions to improve the overall exposition of the paper. Part of this research was supported in the form of graduate Research Assistantships by NSF Grants DMS-1708249 and DMS-2004155. For the second author, this research was partially supported by a Herbert T. Graham Scholarship through the Department of Mathematics at Michigan State University. 
\end{ack*}

%%%%%%%%%%%%%%%%%%%%%%%%%%%%%%%%%%%%%%%%%%%%%%%%%%%%%%%%%%
%%% SECTION 2
%%%%%%%%%%%%%%%%%%%%%%%%%%%%%%%%%%%%%%%%%%%%%%%%%%%%%%%%%%

\section{Link family}\label{AgolandLinksSec}
In this section, we will construct the link family $\mathcal{M}$. We begin by recalling a construction of Agol \cite{AgolSmall} in Subsection \ref{AgolConstrSubSec}. In Subsection \ref{linkfamilySubSec}, we use Agol's algorithm to construct the family of link complements $\mathcal{M}$.

\subsection{Agol's construction of cusped $3$-manifolds}\label{AgolConstrSubSec}
Agol \cite{AgolSmall} introduced a method which uses the pants complex of a surface and links in bundles over that surface to construct compact manifolds with well-understood geometric characteristics. We outline the construction here.

\begin{definition}
Let $\Sigma_{g,n}$ be a connected  compact  orientable surface of genus $g$ with $n$ boundary components and Euler characteristic $\chi(\Sigma_{g,n})= 2(1-g)-n$. We denote the closed surface of genus $g$ by $\Sigma_g$. For $\chi(\Sigma_{g,n})<0$, a \emph{pants decomposition} is a maximal collection of distinct smoothly embedded simple closed curves on $\Sigma_{g,n}$ which have trivial intersection pairwise. A pants decomposition $\{\alpha_1, \dots,\alpha_N\}$ consists of $N = 3(g-1)+n$ curves, and cutting $\Sigma_{g,n}$ along these curves produces $-\chi(\Sigma_{g,n})$ pairs of pants $\Sigma_{0,3}$.
\end{definition}

We note that the pants decompositions of a given surface are not unique.
 
\begin{definition}
Two pants decompositions $P = \{\alpha_1, \dots,\alpha_N\}$ and $P' = \{\alpha_1', \dots,\alpha_N'\}$ of a surface $\Sigma_{g,n}$ are said to \textit{differ by an elementary move} if $P'$ can be obtained from $P$ by replacing one curve $\alpha_i$ with another curve $\alpha_i'$ such that $\alpha_i$ and $\alpha_i'$ intersect \textit{minimally} in one of the following ways:
\begin{itemize}
    \item If $\alpha_i$ lies on a $\Sigma_{1,1}$ in the complement of the other curves in $P$, then $\alpha_i$ is on a single pair of pants and $\alpha_i$ and $\alpha_i'$ must intersect exactly once.
    \item If $\alpha_i$ lies on a $\Sigma_{0,4}$ in the complement of the other curves in $P$, then $\alpha_i$ is the boundary between two pairs of pants and $\alpha_i$ and $\alpha_i'$ must intersect exactly twice.
\end{itemize}
We call a curve switch on $\Sigma_{1,1}$ a \textit{simple move}, or $S-$move, and a curve switch on $\Sigma_{0,4}$ an \textit{associativity move}, or $A-$move. Examples of the elementary moves are given in Figure \ref{fig:AS}.
\end{definition}

\begin{figure}[!htb]
    \centering
    \begin{subfigure}[b]{0.42\textwidth}
        \includegraphics[width=\textwidth]{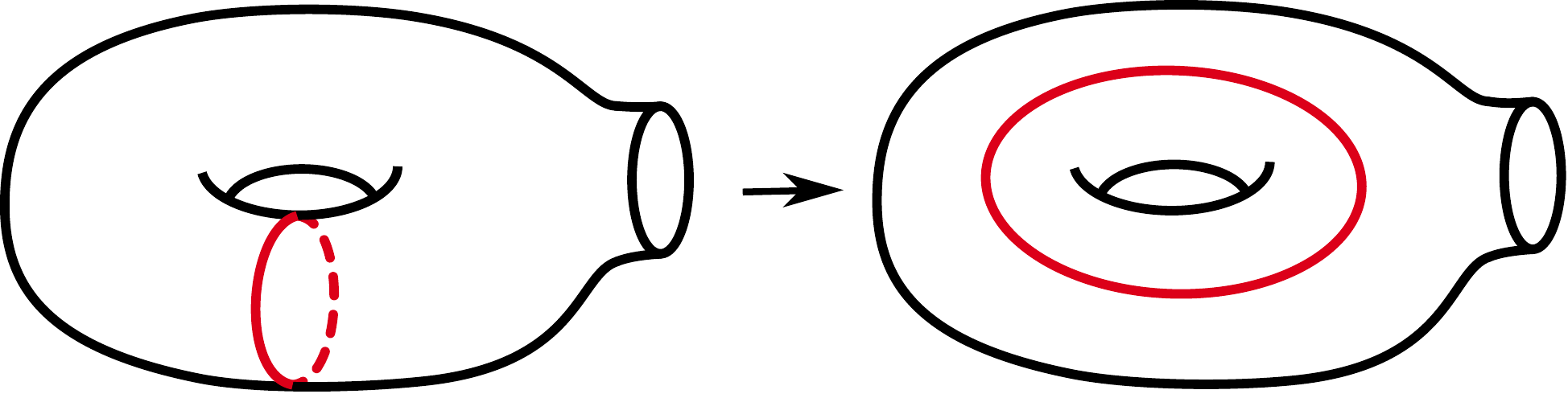}
        \caption{An example of an $S-$move.}
        \label{fig:Smove}
	\end{subfigure}
	\qquad
	\begin{subfigure}[b]{0.42\textwidth}
        \includegraphics[width=\textwidth]{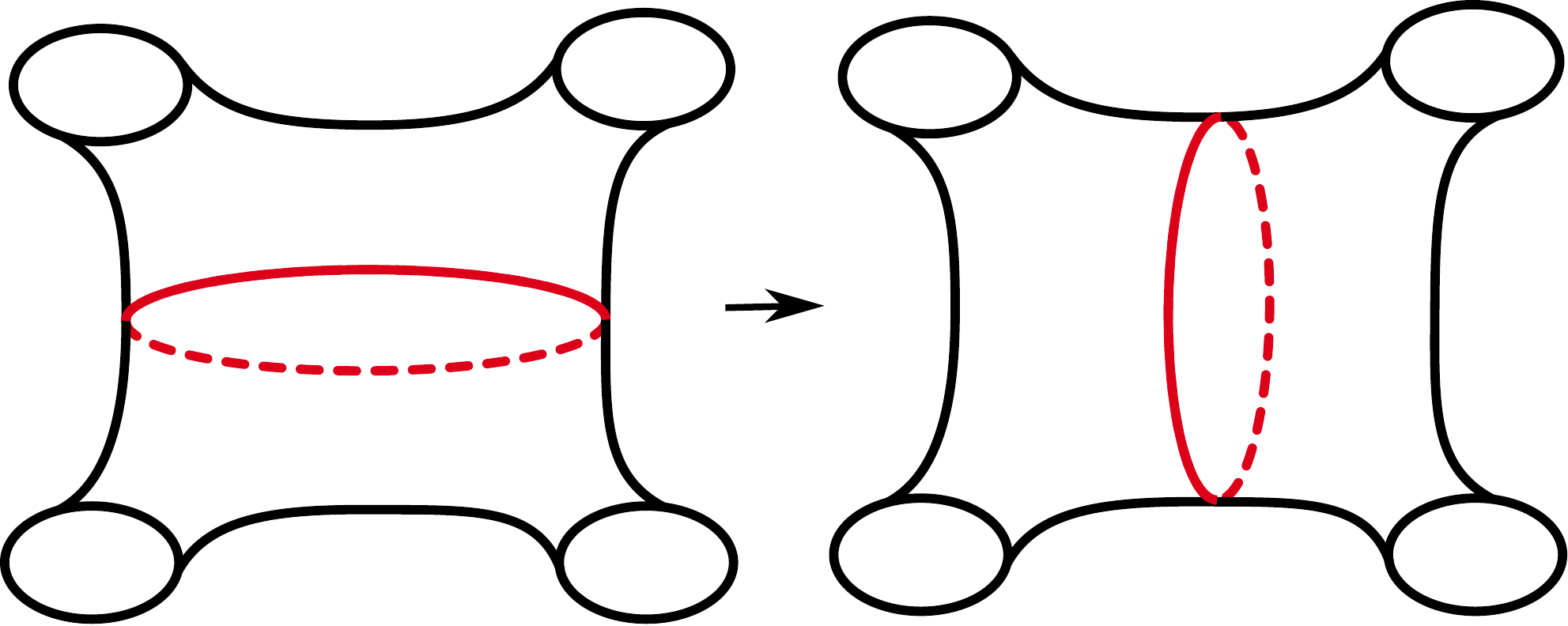}
        \caption{An example of an $A-$move.}
        \label{fig:Amove}
	\end{subfigure}
	\caption{Examples of the elementary moves.}
    \label{fig:AS}
\end{figure}

\begin{definition}
\cite{pantsHLS}
The \emph{pants decomposition graph} $\mathcal{P}(\Sigma_{g,n})^{(1)}$ of $\Sigma_{g,n}$ is the graph with vertices corresponding to isotopy classes of pants decompositions of $\Sigma_{g,n}$ and edges corresponding to pairs of isotopy classes which differ by a single elementary move.
\end{definition}

The following theorem, originally stated at the end of \cite{HATCHERThurstonPants}, is proven by Hatcher, Lochak, and Schneps in Theorem 2 of \cite{pantsHLS}.

\begin{thm}[\cite{pantsHLS}, Theorem 2] 
Let $g,n \in \mathbb{N}\cup\{0\}$ and let $\Sigma_{g,n}$ be a connected compact orientable surface with $\chi(\Sigma_{g,n}) <0$. Then the pants decomposition graph $\mathcal{P}(\Sigma_{g,n})^{(1)}$ is connected.
\end{thm}

\begin{definition}
For a given homeomorphism $f:\Sigma_{g,n} \to \Sigma_{g,n}$, define the associated \emph{mapping torus} by $T_f = (\Sigma_{g,n} \times [0,1]) / ((x,0) \sim (f(x),1))$.
\end{definition}

In \cite{AgolSmall}, Agol constructed cusped $3$-manifolds from the mapping torus $T_f$ and a path $P$ on the pants decomposition graph $\mathcal{P}(\Sigma_{g,n})^{(1)}$. We outline the construction as follows:

\begin{itemize}
    \item Let $f: \Sigma_{g,n} \to \Sigma_{g,n}$ be a homeomorphism and $P= \{P_i\}_{i=0}^m$ be a path such that each $P_i$ is a vertex of the pants decomposition graph, each $P_i$ and $P_{i+1}$ are connected by an edge, and $P_m= f(P_0)$.
    
    \item For $i\in \{1,\dots,m\}$, let $\beta_i$ correspond to the simple closed curve in  $P_i$ obtained from performing a single elementary move on a simple closed curve in $P_{i-1}$. We assume there exists no curve $\beta_j$ that is contained in all the pants decompositions $P_i$.

    \item Let $B= \{B_i\}_{i=1}^m$ be the link in $T_f$ such that $B_i = \beta_i \times \{\frac{i}{m}\}$ is a link component, and we define the cusped $3$-manifold $M_P$ to be the complement of the link $B$ in $T_f$.
\end{itemize}
Agol proves the following lemma in \cite{AgolSmall}.

\begin{lem} [\cite{AgolSmall}, Lemma 2.3 and Corollary 2.4]\label{LemAgolVol} 
Let $M_P$ be the cusped $3$-manifold obtained from Agol's construction for a homeomorphism $f: \Sigma_{g,n} \to \Sigma_{g,n}$ and a path $P$ on $\mathcal{P}(\Sigma_{g,n})^{(1)}$. Then $M_P$ has a complete hyperbolic metric such that $vol(M_P) = (|S|+2|A|)v_8$ where $vol(M_P)$ is the hyperbolic volume, $|S|$ and $|A|$ are the number of $S-$ and $A-$moves in $P$, respectively, and $v_8 \approx 3.66$ is the volume of a regular ideal hyperbolic octahedron. 
\end{lem}

\subsection{Manifold construction}\label{linkfamilySubSec}
Here we will discuss a family of links with octahedral complements in $S^1$-bundles over connected closed orientable surfaces. 

Let $\Sigma_{g}$ be a connected closed orientable surface of genus $g$ constructed by gluing $k$ copies of $\Sigma_{1,1}$ and $l$ copies of $\Sigma_{0,4}$  along their boundary components. We glue each pair of $S^1$ boundary components via identity maps. Since $\Sigma_g$ has no boundary components, $k$ must be even. 

Consider the closed orientable $3$-manifold $T_{id} = \Sigma_{g} \times S^1$. Note each gluing circle of $\Sigma_g$ corresponds to a torus in $T_{id}$. This manifold can be decomposed into elementary pieces by cutting along these tori so that the resulting pieces are trivial $S^1$-bundles over $k$ copies of  $\Sigma_{1,1}$ and $l$ copies of $\Sigma_{0,4}$. We perform a pair of $S-$moves in each copy of $\Sigma_{1,1}$ with $P_0 = P_2$ on the pants complex of $\Sigma_{1,1}$ to produce a two-component link $L_S$. By Lemma \ref{LemAgolVol}, the complement of $L_S$ in $\Sigma_{1,1} \times S^1$ has a complete hyperbolic metric with hyperbolic volume $2v_8$. We call this complement an \textit{$S-$piece}. Similarly, we perform a pair of $A-$moves in each copy of $\Sigma_{0,4}$ with $P_0 = P_2$ to produce a two-component link $L_A$. By Lemma \ref{LemAgolVol}, the complement of $L_A$ in $\Sigma_{0,4} \times S^1$ has a complete hyperbolic metric with hyperbolic volume $4v_8$. We call this complement an \textit{$A-$piece}.

Let $L = \bigsqcup_{i=1}^k L_S^i \cup \bigsqcup_{j=1}^l L_A^j$ be the union of these two component links in $\Sigma_{g} \times S^1$. We denote the $(2k+2l)$-component link complement $\left(\Sigma_{g} \times S^1\right) \backslash L$ by $M_L(k,l)$. We remark that $M_L(k,l)$ is not hyperbolic since the gluing procedure produces essential tori. However, by Lemma \ref{LemAgolVol}, each $S-$piece and $A-$piece of $M_L(k,l)$ contributes $2v_8$ and $4v_8$ to the simplicial volume, respectively, so $v_3 \|M_L(k,l)\| = 2(k+2l)v_8$. Two examples of manifolds of type $(2,2)$ are given in Figure \ref{fig: ManifoldEx}. Figure \ref{fig: SplitManifoldEx} gives a decomposition of each example into their respective $S-$ and $A-$pieces.

\begin{figure}[!htb]
    \centering
    \begin{subfigure}[b]{0.32\textwidth}
	    \includegraphics[width=\textwidth]{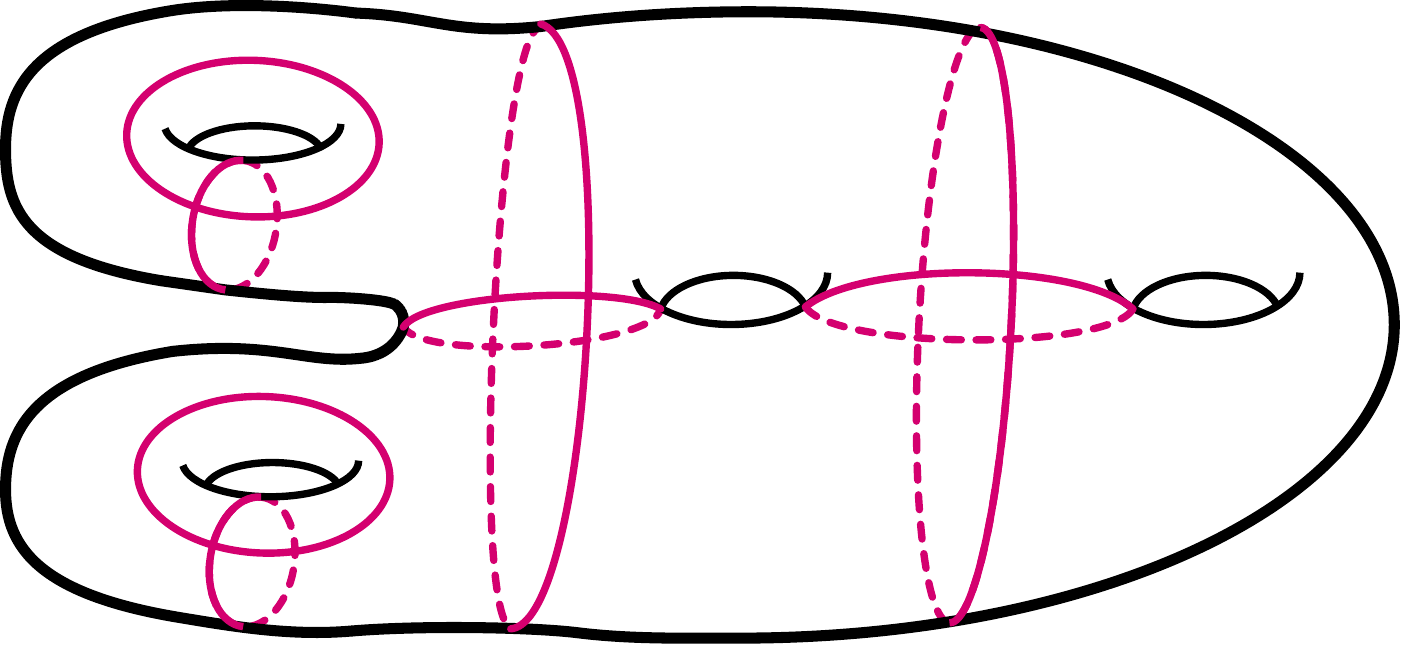}
	\end{subfigure}
	\qquad \qquad
	\begin{subfigure}[b]{0.32\textwidth}
	    \includegraphics[width=\textwidth]{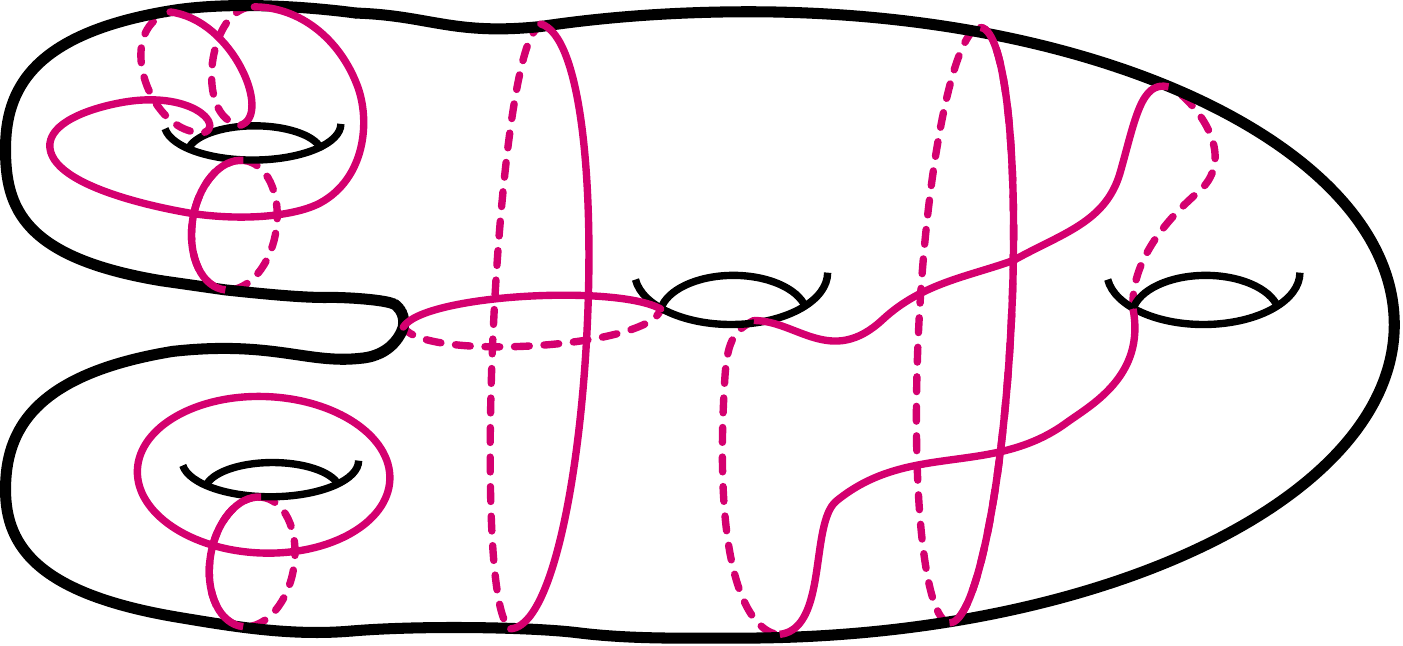}
	\end{subfigure}
    \caption{The projections of $M_L(2,2)$ and $M_{L'}(2,2)$ onto the base surface $\Sigma_4$, where $L$ and $L'$ are 8-component links in $\Sigma_4 \times S^1$.}
    \label{fig: ManifoldEx}
\end{figure}

\begin{figure}[!htb]
    \centering
    \begin{subfigure}[b]{0.8\textwidth}
    	\centering
    	\includegraphics[width=0.9\textwidth]{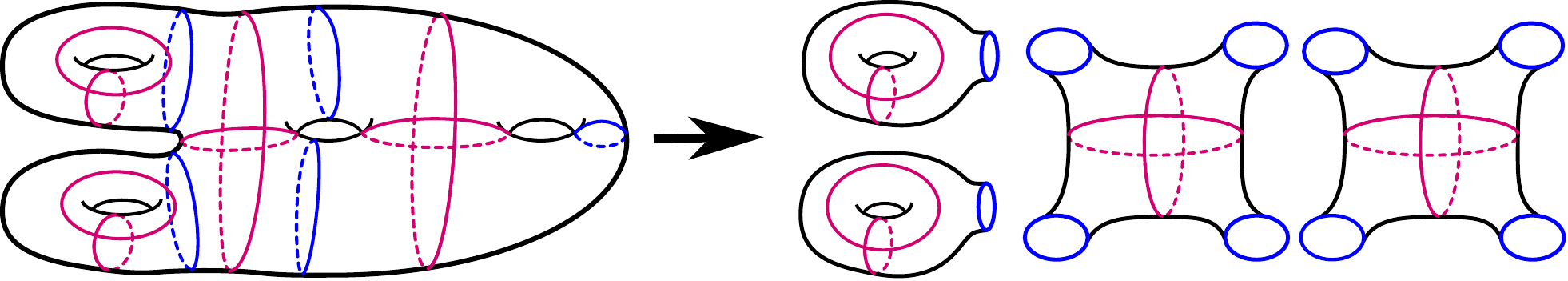}\\
    	\vspace{0.25cm}
    \end{subfigure}
    \begin{subfigure}[b]{0.8\textwidth}
    	\centering
    	\includegraphics[width=0.9\textwidth]{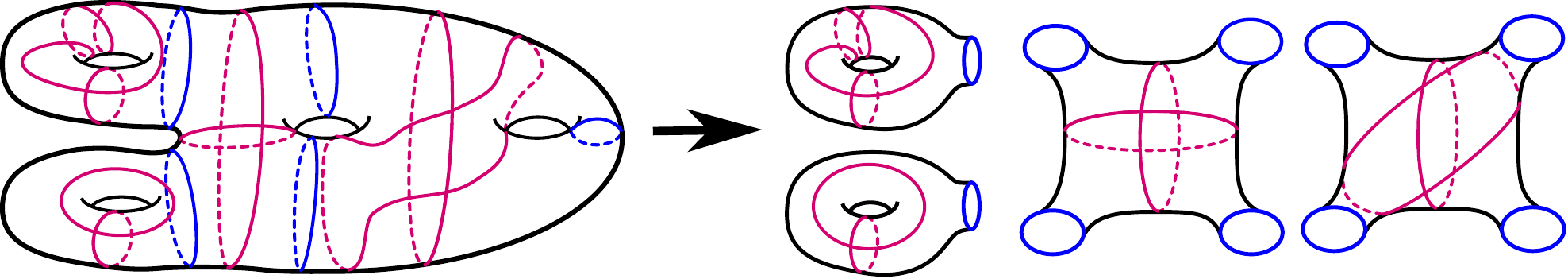}
    \end{subfigure}
    \caption{From cutting $\Sigma_4$ along the blue curves which lift to essential tori in $M_L(2,2)$ and $M_L'(2,2)$, we obtain two $S-$pieces and two $A-$pieces.}
    \label{fig: SplitManifoldEx}
\end{figure}

Let $\mathcal{M} = \{M_L(k,l) \: | \: L \subset \Sigma_g \times S^1,\; g \geq 2, \; k,l \in \mathbb{N}, \; k \text{ even} \}$ be the family of compact orientable $3$-manifolds constructed from $k$ $S-$pieces and $l$ $A-$pieces. In Section \ref{mainThmProofSec}, we will prove Theorem \ref{mainthm} for manifolds in this infinite family.

\begin{remark}
Note that since we only require that $P_0=P_2$ in the construction of the $S-$ and $A-$pieces, we can take $P_0$ to be an arbitrary vertex on the pants decomposition graph which gives rise to infinitely many choices for $P_1$. This implies that we also have infinitely many choices for the elementary pieces used in the construction of $M_L(k,l)$. That being said, because they are hyperbolic,  Corollary 6.6.2 by Thurston \cite{ThurstonGT3manifolds} implies that there are at most finitely many elementary pieces up to homeomorphism.
\end{remark}

%%%%%%%%%%%%%%%%%%%%%%%%%%%%%%%%%%%%%%%%%%%%%%%%%%%%%%%%%%
%%% SECTION 3
%%%%%%%%%%%%%%%%%%%%%%%%%%%%%%%%%%%%%%%%%%%%%%%%%%%%%%%%%%

\section{The Turaev--Viro invariants from the perspective of shadows}\label{TVfromShadowSec}

In this section, we introduce Turaev's shadow theory of $3$-manifolds. We begin with an introduction to the quantum 6j-symbols and some of their properties in Subsection \ref{Q6jSubSec}. In Subsection \ref{Shadowsubsec}, we introduce Turaev's shadow state sum invariant for links in $S^1$-bundles over surfaces and briefly discuss how Turaev's shadow construction can be generalized to all $3$-manifolds. In Subsection \ref{RelRTsubsec}, we introduce the relative Reshetikhin--Turaev invariants and relate them to the shadow state sum invariant. We then introduce the Turaev--Viro invariants and discuss the results of Belletti, Detcherry, Kalfagianni, and Yang \cite{growth6j} relating the them to the relative Reshetikhin--Turaev invariants.

For the rest of this paper, let $r\geq 3$ be an odd integer and $q = e^{\frac{2\pi \sqrt{-1}}{r}}$. Define the \textit{quantum integer} by $[n] = \frac{q^n-q^{-n}}{q-q^{-1}}$ and the quantum factorial by
\[
[n]! = \prod_{k=1}^n [k]
\]
where $k$ is a positive integer. By convention, we also define $[0]!=1$.
Finally, let $I_r = \{0,1,2,\dots,r-2\}$. Throughout the paper, we use the convention that $\sqrt{y} = \sqrt{|y|}\sqrt{-1}$ for any negative real number $y$.

\subsection{Quantum 6j symbols}\label{Q6jSubSec}
We first introduce the quantum $6j$-symbols. Deeper algebraic and geometric properties of the quantum $6j$-symbols can be found in Kirillov and Reshetikhin \cite{KirillovReshetikhin6j}, Turaev and Viro \cite{TuraevViro}, and Turaev \cite{TuraevShadow,TuraevBook}. Note that other than the use of $[n]$ for the quantum integer rather than $\{n\} : = q^n - q^{-n}$, the definitions introduced in this section largely follow the conventions and notation of \cite{growth6j}. In particular, our coloring set $I_r$ is defined in terms of integers rather than the conventionally chosen half-integers used in \cite{chen2018volume, TuraevViro}.

\begin{definition}\label{rAdmissible}
A triple $(a_1,a_2,a_3)$ of integers in $I_r$ is \textit{$r$-admissible} if 
\begin{enumerate}[(i)]
    \item $a_1+a_2+a_3$ is even,
    \item $a_1+a_2+a_3 \leq 2(r-2)$,
    \item $a_i+a_j-a_k \geq 0$ for any $i,j,k \in \{1,2,3\}$.
\end{enumerate}
We say a 6-tuple $(a_1,\dots,a_6)$ is \textit{$r$-admissible} if the triples $(a_1,a_2,a_3)$, $(a_1,a_5,a_6)$, $(a_2,a_4,a_6)$, and $(a_3,a_4,a_5)$ are $r$-admissible.
\end{definition}
For an $r$-admissible triple $(a_1,a_2,a_3)$, define 
\[
\Delta (a_1,a_2,a_3) = \sqrt{\frac{\left[\frac{a_1+a_2-a_3}{2}\right]! \left[\frac{a_1+a_3-a_2}{2}\right]! \left[\frac{a_2+a_3-a_1}{2}\right]!}{\left[\frac{a_1+a_2+a_3}{2} +1 \right]!}}
.\]

\begin{definition}\label{q6jdef}
The \textit{quantum 6j-symbol} of an $r$-admissible 6-tuple $(a_1,\dots,a_6)$ is the complex number
\begin{align}\label{q6j}
    \left| \begin{array}{ccc}
    a_1 & a_2 & a_3 \\
    a_4 & a_5 & a_6
    \end{array} \right| =& \sqrt{-1}^{-\sum_{i=1}^6 a_i} \Delta(a_1,a_2,a_3)\Delta(a_1,a_5,a_6)
    \Delta(a_2,a_4,a_6)\Delta(a_3,a_4,a_5)  \nonumber\\
    & \sum_{k = \text{max}\{T_i\}}^{\text{min}\{Q_j\}} \frac{(-1)^k [k+1]!}{\prod_{i=1}^4 \left[k-T_i\right]! \prod_{j=1}^3 \left[Q_j-k\right]!} \in \mathbb{C},
\end{align}
where $T_1 = \frac{a_1+a_2+a_3}{2}$, $T_2 = \frac{a_1+a_5+a_6}{2}$, $T_3 = \frac{a_2+a_4+a_6}{2}$, $T_4 = \frac{a_3+a_4+a_5}{2}$, $Q_1 = \frac{a_1+a_2+a_4+a_5}{2}$, $Q_2 = \frac{a_1+a_3+a_4+a_6}{2}$, and $Q_3 = \frac{a_2+a_3+a_5+a_6}{2}$.
\end{definition}
We remark that the value of the quantum $6j$-symbol is either real or purely imaginary. 

We now recall some properties of the quantum $6j$-symbol at $q = e^{\frac{2\pi \sqrt{-1}}{r}}$. 
For an $r$-admissible $6$-tuple $(i,j,k,l,m,n)$, the symmetries 
\begin{align}\label{q6jSymmetries}
     \left| \begin{array}{ccc}
    i & j & k \\
    l & m & n
    \end{array} \right| = 
    \left| \begin{array}{ccc}
    j & i & k \\
    m & l & n
    \end{array} \right| = 
     \left| \begin{array}{ccc}
    i & k & j \\
    l & n & m
    \end{array} \right| =  
    \left| \begin{array}{ccc}
    i & m & n \\
    l & j & k
    \end{array} \right| = 
     \left| \begin{array}{ccc}
    l & m & k \\
    i & j & n
    \end{array} \right| = 
     \left| \begin{array}{ccc}
    l & j & n \\
    i & m & k
    \end{array} \right|
\end{align}
follow immediately from the definition of the quantum $6j$-symbol.

Belletti, Detcherry, Kalfagianni, and Yang \cite{growth6j} give an upper bound for the growth rate of the quantum $6j$-symbol, which we state in the following theorem. Related results on these growth rates are also due to Costantino \cite{Costantino6j2007}.

\begin{thm}[\cite{growth6j}, Theorem 1.2 and Lemma 3.13]\label{growthrate6jThm}
For any odd $r\geq 3$ and any $r$-admissible 6-tuple $(a_1,\dots,a_6)$, 
\begin{align}\label{upperbound6j}
    \frac{2\pi}{r} \log \left|\left| \begin{array}{ccc}
    a_1 & a_2 & a_3 \\
    a_4 & a_5 & a_6
    \end{array} \right|_{q = e^{\frac{2\pi \sqrt{-1}}{r}}} \right| \leq v_8 + O\left(\frac{\log(r)}{r}\right). 
\end{align}
Moreover, this bound is sharp. If the sign is chosen such that $\frac{r \pm 1}{2}$ is even, then
\begin{align}\label{sharp6j}
    \frac{2\pi}{r} \log \left|\left| \begin{array}{ccc}
    \frac{r \pm 1}{2} & \frac{r \pm 1}{2} & \frac{r \pm 1}{2} \\
    \frac{r \pm 1}{2} & \frac{r \pm 1}{2} & \frac{r \pm 1}{2}
    \end{array} \right|_{q = e^{\frac{2\pi \sqrt{-1}}{r}}} \right| = v_8 +O\left(\frac{\log(r)}{r}\right). 
\end{align}
\end{thm}

\iffalse

The authors of \cite{growth6j} also prove the following result of Costantino \cite{Costantino6j2007} for the root $q = e^{\frac{2\pi \sqrt{-1}}{r}}$ which is contained in the proof of Theorem A.1 of \cite{growth6j}.

\begin{lem}[\cite{growth6j}, Theorem A.1]\label{Sksign}
Let $(a_1,\dots, a_6 )$ be an admissible 6-tuple such that 
\begin{enumerate}[(1)]
\item $0 \leq Q_j-T_i \leq \frac{r-2}{2}$ for $i=1,2,3,4$ and $j=1,2,3$, and 
\item $\frac{r-2}{2}\leq T_i \leq r-2$ for $i=1,2,3,4$. 
\end{enumerate}
Then the sign of $S_k$ is independent of the choice of k, for $k = \max{\{T_i\}},\dots, \min{\{Q_j\}}$. 
\end{lem}

\fi
%%%%%%%%%%%%%%%%%%%%%%%%%%%%%%%%%%%%%%%%%%%%%%%%%%%%%%
%% Statement of full Theorem A.1
%%%%%%%%%%%%%%%%%%%%%%%%%%%%%%%%%%%%%%%%%%%%%%%%%%%%%

%\iffalse

The authors of \cite{growth6j} also prove the following result of Costantino \cite{Costantino6j2007} for the root $q = e^{\frac{2\pi \sqrt{-1}}{r}}$. Let the summand of Equation (\ref{q6j}) be given by 
\[
S_k = \frac{(-1)^k [k+1]!}{\prod_{i=1}^4 \left[k-T_i\right]! \prod_{j=1}^3 \left[Q_j-k\right]!}
\]
where $k \in \{\max{\{T_i\}},\dots, \min{\{Q_j\}} \}$ for $i=1,2,3,4$ and $j=1,2,3$.

\begin{thm}[\cite{growth6j}, Theorem A.1]\label{growthtrunctetra}
Let $(a_1^{(r)},\dots, a_6^{(r)} )$ be a sequence of  admissible 6-tuples such that 
\begin{enumerate}[(a)]
\item $0 \leq Q_j-T_i \leq \frac{r-2}{2}$ for $i=1,2,3,4$ and $j=1,2,3$, and 
\item $\frac{r-2}{2}\leq T_i \leq r-2$ for $i=1,2,3,4$. 
\end{enumerate}
Let $\theta_i = \lim_{r\rightarrow \infty} \frac{2\pi a_i^{(r)}}{r}$ and let $\alpha_i = |\pi-\theta_i|$. Then 
\begin{enumerate}[(1)]
    \item for each $r$, the sign of $S_k$ is independent of the choice of k, for $k \in \{\max{\{T_i\}},\dots, \min{\{Q_j\}} \}$, 
    \item $\alpha_1,\dots, \alpha_6$ are the dihedral angles of an ideal or a hyperideal hyperbolic tetrahedron $\Delta$, see Remark \ref{hyperideal}, and 
    \item as $r$ runs over the odd integers
    \begin{align}\label{trunctetra}
        \lim_{r\rightarrow\infty} \frac{2\pi}{r} \log \left|\left| \begin{array}{ccc}
        a_1^{(r)} & a_2^{(r)} & a_3^{(r)}\\
        a_4^{(r)} & a_5^{(r)} & a_6^{(r)} 
        \end{array} \right|_{q = e^{\frac{2\pi \sqrt{-1}}{r}}} \right| = vol(\Delta)
    \end{align}
\end{enumerate}
\end{thm}

\begin{remark}\label{hyperideal}
We refer to \cite{BaoBonahon} for details on hyperideal hyperbolic tetrahedra. In particular, the numbers $\alpha_1,\dots, \alpha_6$ correspond to the dihedral angles of an ideal or hyperideal hyperbolic tetrahedron if and only if around each vertex, $\alpha_i+\alpha_j+ \alpha_k \leq \pi$ for $i,j,k \in \{1,\dots, 6\}$. 
\end{remark}

The even integers can be written as the two sets $\{ \frac{r+1}{2} \: | \: r \equiv 3 \mod 4\}$ and $\{ \frac{r-1}{2} \: | \: r \equiv 1 \mod 4\}$ corresponding to the subsequences that achieve the sharp upper bound of Theorem \ref{growthrate6jThm} Equation (\ref{sharp6j}). Another such pair of subsequences is $\left( \frac{r-3}{2}\right)$ and $\left( \frac{r-1}{2}\right)$. The following is analogous to Lemma 3.13 of \cite{growth6j}.

\begin{lem}\label{modifiedq6jbound}
If the sign is chosen such that $\frac{r -2 \pm 1}{2}$ is even, then
\begin{align}\label{modifiedsharp6j}
    \frac{2\pi}{r} \log \left|\left| \begin{array}{ccc}
    \frac{r -2 \pm 1}{2} & \frac{r -2 \pm 1}{2} & \frac{r -2 \pm 1}{2} \\
    \frac{r -2 \pm 1}{2} & \frac{r -2 \pm 1}{2} & \frac{r -2 \pm 1}{2}
    \end{array} \right|_{q = e^{\frac{2\pi \sqrt{-1}}{r}}} \right| = v_8 +O\left(\frac{\log(r)}{r}\right). 
\end{align}
\end{lem}

\begin{proof}
First note that the $r \equiv 1 \mod 4$ case is covered by Equation (\ref{sharp6j}). When $r \equiv 3 \mod 4$, $T_i = \frac{3(r-3)}{4}$ for all $i=1,2,3,4$ and $Q_j = r-3$ for $j=1,2,3$, so the $6$-tuple $\left(\frac{r-3}{2},\dots,\frac{r-3}{2} \right)$ satisfies the assumptions of Theorem \ref{growthtrunctetra} for $r \geq 5$. Here the corresponding hyperideal truncated tetrahedron $\Delta$ has dihedral angles $\alpha_i = 0$ for all $i$, so $\Delta$ is a regular ideal hyperbolic octahedron and $vol(\Delta) = v_8$. We refer to \cite{Costantino6j2007}, Definition 2.1 for details. By part $(3)$ of Theorem \ref{growthtrunctetra}, \begin{align*}
   \lim_{r\rightarrow \infty} \frac{2\pi}{r} \log \left|\left| \begin{array}{ccc}
    \frac{r -3}{2} & \frac{r -3}{2} & \frac{r -3}{2} \\
    \frac{r -3}{2} & \frac{r -3}{2} & \frac{r -3}{2}
    \end{array} \right|_{q = e^{\frac{2\pi \sqrt{-1}}{r}}} \right| = v_8.
\end{align*}
\end{proof}

%\fi
%%%%%%%%%%%%%%%%%%%%%%%%%%%%%%%%%%%%%%%%%%%%%%%%%%%%

%\textcolor{green}{
%\subsection{Relating the quantum invariants}\label{RelRTsubsec}
%We now describe Turaev's state sum invariants for two-dimensional polyhedra representing links in $S^1$-bundles over surfaces, and we relate them to the $r$-th relative Reshetikhin--Turaev and Turaev--Viro invariants. In an effort  to construct analogous invariants to the colored Jones polynomial of links in $S^3$, Turaev \cite{TuraevShadow,TuraevBook} introduces a technique to present links in $S^1$-fibrations over surfaces as loops on $\Sigma_{g}$ with additional topological data given by the bundle. From this $2$-dimensional presentation, we can build quantum invariants of the colored link. In \cite{TuraevBook}, Turaev relates these invariants to the $r$-th relative Reshetikhin--Turaev invariants  which can be used to compute the Turaev--Viro invariants \cite{growth6j, colJvolDKY}.}

\subsection{Shadow state sum invariants}\label{Shadowsubsec}

We now describe Turaev's state sum invariants for two-dimensional polyhedra representing links in $S^1$-bundles over surfaces. In an effort  to construct analogous invariants to the colored Jones polynomial of links in $S^3$, Turaev \cite{TuraevShadow,TuraevBook} introduces a technique to present links in $S^1$-fibrations over surfaces as loops on $\Sigma_{g}$ with additional topological data given by the bundle. From this $2$-dimensional presentation, we can build quantum invariants of the colored link.

We begin by recalling the construction of Turaev's shadow state sum invariant \cite{TuraevShadow, TuraevBook} for $S^1$-bundles over surfaces, largely following the construction given in \cite{TuraevShadow}. Let $\Sigma_{g,n}$ be a compact orientable surface of genus $g$ with $n$ boundary components. Consider a finite collection of loops $\{l_i : S^1 \rightarrow \Sigma_{g,n}\}$ on $\Sigma_{g,n}$ with only double transversal crossings $l_i \cap l
_j$ for any $i,j$. Denote by $\Gamma$ the 1-dimensional CW-complex consisting of the collection of loops $\{l_i\}$ and crossing points $\{l_i \cap l_j\}$, and let $P$ denote the pair $(\Sigma_{g,n}, \Gamma)$. We define the connected components $X_t$ of $\Sigma_{g,n} \backslash \Gamma$ to be the \textit{regions} of $P$.

\begin{definition}
A \textit{shadow} is a pair $(P,gl)$ where $gl: \{X_t\} \rightarrow \frac{1}{2}\mathbb{Z}$ is a map that assigns a half-integer to each region of $P$. This half-integer is called the \textit{gleam} of the region. The \textit{total gleam} of a shadow is defined to be
\begin{align*}
    \text{total gleam} &= \sum_t \left(gl(X_t)\right) -2 \# \{l_i \cap l_j\},
\end{align*}
where $\# \{l_i \cap l_j\}$ is the number of crossing points of $P$.
\end{definition}

We will restrict our attention to shadows on closed surfaces. Suppose $\Sigma_{g}$ is a closed orientable surface and $\rho:M \rightarrow \Sigma_{g}$ is an oriented $S^1$-bundle over $\Sigma_{g}$. Let $L\subset M$ be a link. We say that $L\subset M$ is \textit{generic} if it is transverse to the fibers with respect to $\rho$ and the collection of immersed loops $\rho(L) \subset \Sigma_{g,n}$ only have double transversal crossings. In \cite{TuraevShadow}, Turaev constructs a map which associates a shadow $(P(L),gl)$ to $L \subset M$, where the gleams of each region of $P(L)$ are determined by the Euler number of the 2-dimensional real vector bundle associated to the oriented circle bundle $\rho$. The construction of the map $gl: \{X_t\} \rightarrow \frac{1}{2}\mathbb{Z}$ for general $S^1$-bundles over $\Sigma_g$ will not be relevant to the arguments that follow, so we refer to Section 3(a) of \cite{TuraevShadow} for further details. The following theorem of Turaev is a result of this construction. 

\begin{thm}[\cite{TuraevShadow}, Theorem 3.2]\label{canonShadow}
Let $\rho:M \rightarrow \Sigma_{g}$ be an oriented circle bundle over a closed orientable surface $\Sigma_{g}$, and let $L \subset M$ be a generic link with respect to $\rho$. Then there is a shadow $(P(L), gl)$ with total gleam $-\chi(p)$ associated to $L \subset M$, where $\chi(p)$ is the Euler number of the bundle.
\end{thm}

For our purposes, we restrict further to the special case where $L \subset \Sigma_{g} \times S^1$ and $\rho: \Sigma_{g} \times S^1 \rightarrow \Sigma_g$ is the trivial bundle. We can embed $\Sigma_g \times [0,1] \hookrightarrow \Sigma_g \times S^1$ via the map 
\[
(x,t) \longmapsto \left(x,e^{2\pi \sqrt{-1}t}\right).
\]
Now consider $L$ as a subset of $\Sigma_g \times [0,1]$. It is a generic link with well-defined over- and under-crossings in the projection $\rho|_{\Sigma_g \times [0,1]}(L)$ on $\Sigma_g$. This projection produces a shadow on $\Sigma_g$ with gleams assigned as in Figure \ref{fig: GleamTrivial}, where the gleam of each region is the sum of the associated $1$'s.

\begin{figure}[!htb]
    \centering
    \includegraphics[width=3cm]{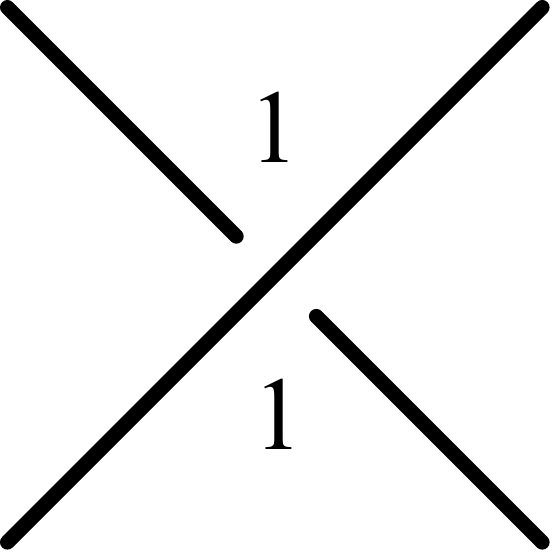}
    \caption{Gleam assignment for shadows of links in trivial bundles $\Sigma_{g} \times [0,1]$.}
    \label{fig: GleamTrivial}
\end{figure}

Note that $L \subset \Sigma_g \times [0,1] \subset \Sigma_g \times S^1$ is a generic link with respect to $\rho$. The projection $\rho|_{\Sigma_g \times [0,1]}(L)$ with gleams assigned using Figure \ref{fig: GleamTrivial} coincides with the shadow $(P(L),gl)$ constructed using Theorem \ref{canonShadow} with total gleam $-\chi(\rho) = 0$. As an elementary example, we consider the following example of Costantino and Thurston \cite{CostantinoThurston}. 

\begin{ex}[\cite{CostantinoThurston}, Example 3.5]
Consider the shadow on $\Sigma_0 = S^2$ with $\Gamma = \emptyset$ and total gleam $0$. This shadow corresponds to the empty link in the bundle $\rho: S^2 \times S^1 \rightarrow S^2$, where the triviality of the bundle is encoded by the zero gleam.
\end{ex}

In order to define Turaev's shadow state sum invariant, we need to consider colorings of the link $L \subset M$. 
\begin{definition}
Let $M$ be a closed $3$-manifold. An \textit{$I_r$-coloring} of a link $L \subset M$ assigns an element of $I_r$ to each component of $L$. Similarly, an \textit{$I_r$-coloring} of a shadow $(P,gl)$ assigns an element of $I_r$ to each loop of $(P,gl)$.
\end{definition}
If $\rho:M \rightarrow \Sigma_{g}$ is a circle bundle over a closed surface $\Sigma_{g}$, an $I_r$-coloring $\gamma$ of a link $L$ in $M$ descends to an $I_r$-coloring $\gamma$ of the loops of the shadow $(P(L),gl)$ constructed using Theorem \ref{canonShadow}.

\begin{definition}
Let $(P,gl,\gamma)$ be an $I_r$-colored shadow with gleams $gl$ and loops colored by $\gamma$. A \textit{surface-coloring} $\eta$ of $(P,gl,\gamma)$ assigns an element of $I_r$ to each region of $(P,gl,\gamma)$.
\end{definition}
Suppose an edge $e$ from a loop of $(P,gl,\gamma)$ is adjacent to two regions $X,X'$ of $(P,gl,\gamma)$. This edge has a fixed color $\gamma(e)$, and the regions $X$ and $X'$ are assigned colors $\eta(X)$ and $\eta(X')$, respectively, by $\eta$. 

\begin{definition}
A surface-coloring $\eta$ of $(P,gl,\gamma)$ is called \textit{admissible} if for any edge $e$ adjacent to two regions $X,X'$ of $(P,gl,\gamma)$, the triple $(\gamma(e), \eta(X),\eta(X')) \in I_r^3$ is $r$-admissible in the sense of Definition \ref{rAdmissible}.
\end{definition}
Figure \ref{fig: AdmEdge} gives the local picture for admissibility. Let adm$(P,gl,\gamma)$ denote the set of admissible surface-colorings of $(P,gl,\gamma)$.

\begin{figure}[!htb]
    \centering
    \includegraphics[width=6cm]{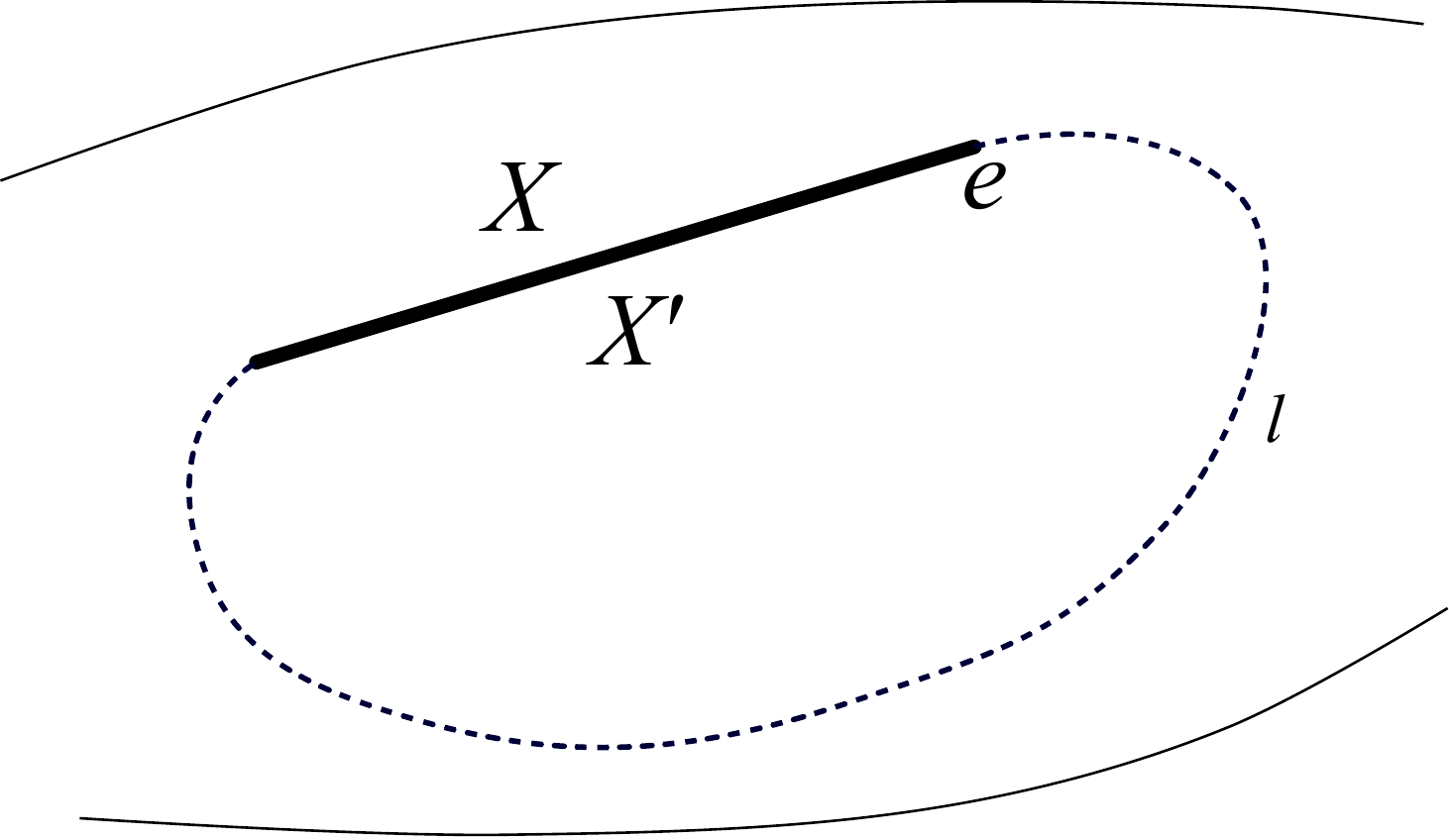}
    \caption{An edge $e$ contained in a loop $l$ of $(P,gl,\gamma)$ and its two adjacent regions $X$ and $X'$. An admissible surface-coloring assigns colors $\eta(X)$ and $\eta(X')$ for which $(\gamma(e), \eta(X),\eta(X'))$ an $r$-admissible triple.}
    \label{fig: AdmEdge}
\end{figure}

Suppose $c_1, \dots, c_p$ are the crossing points of $P$, each an intersection of two distinct loops or a self-crossing of a single loop of $P$. Suppose these loops have colors $i$ and $l$, respectively. Then an admissible surface-coloring $\eta \in \text{adm}(P,gl,\gamma)$ assigns colors $j,k,m,n$ to the four regions incident at the crossing point $c_s$ so that $(i,j,k,l,m,n)$ forms an $r$-admissible $6$-tuple. In particular, $(i,j,k), (i,m,n), (j,l,n),$ and $(k,l,m)$ are $r$-admissible triples. Figure \ref{fig: admAreaColoring} illustrates an admissible surface-coloring $(j,k,m,n)$ around a crossing point of loops colored by $i$ and $l$.

\begin{figure}[!htb]
    \centering
    \includegraphics[width=3cm]{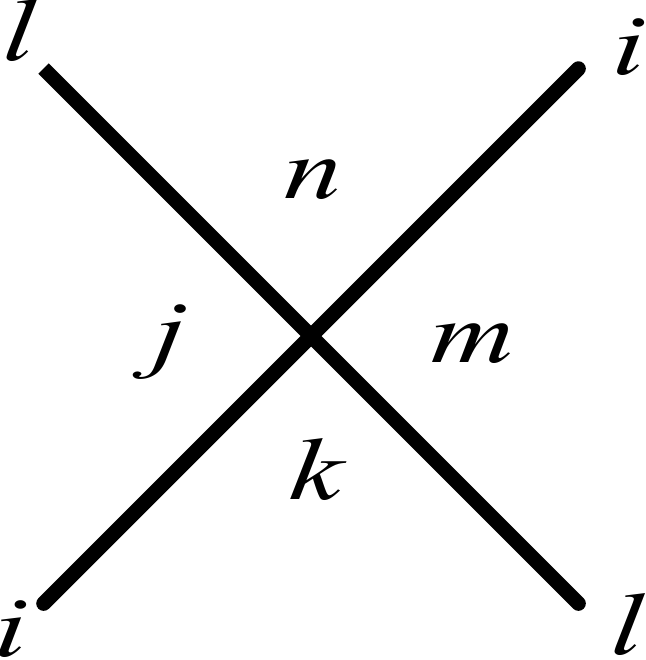}
    \caption{Admissible surface-coloring at a crossing.}
    \label{fig: admAreaColoring}
\end{figure}
Using Definition \ref{q6jdef} of the quantum $6j$-symbol, we let
\[
|c_s|^{\eta} = \left| \begin{array}{ccc}
    i & j & k \\
    l & m & n
    \end{array} \right| \in \mathbb{C}.
\]

Let $X_1,\dots, X_q$ be the regions of $(P,gl,\gamma)$, and let $x_t$, $\chi_t,$ and $z_t$ be the gleam, Euler characteristic, and number of corners of the region $X_t$, respectively. Define the \textit{modified gleam} of $X_t$ by $x_t' = x_t-z_t/2$.
For $j \in I_r$, let
\begin{align*}
    u_j = \pi \sqrt{-1}\left(\frac{j}{2}\right)\left(1-\frac{j+2}{r}\right), & \qquad v_j = (-1)^j[j+1].
\end{align*}
Then for each admissible surface-coloring $\eta$, let
\begin{align}\label{state}
|(P,gl)|^{\eta}_{\gamma} &= \prod_{s = 1}^p |c_s|^{\eta} \prod_{t=1}^q  \left(\left(v_{\eta(X_t)}\right)^{\chi_t} \text{exp}\left(2 u_{\eta(X_t)} x_t'\right)  \right) \in \mathbb{C},
\end{align}
where $\eta(X_t)$ is the region color of $X_t$ assigned by $\eta$.

\begin{remark}
Note that the gleams and the surface-colorings are independent of each other. The gleams encode topological data from the $S^1$-bundle $\rho:M \rightarrow \Sigma_{g}$ and do not affect the quantum $6j$-symbols in the first product of Equation (\ref{state}), only the second product taken over the regions of $(P,gl)$. 
\end{remark}

\begin{definition}
The \textit{shadow state sum} $|(P,gl)|_{\gamma}$ is defined by the following sum over all admissible surface-colorings $\eta \in$ adm$(P,gl,\gamma)$.
\begin{align}\label{statesum}
    |(P,gl)|_{\gamma} &:= \sum_{\eta \in \text{adm}(P,gl,\gamma)} \left|(P,gl)\right|_{\gamma}^{\eta} \in \mathbb{C}.
\end{align}
\end{definition}

Turaev established the following theorem in \cite{TuraevShadow} and generalized it in \cite{TuraevBook}.

\begin{thm}[\cite{TuraevShadow}, Theorem 5.1 and Corollary 5.2]\label{Theorem:StateSum}
Let $(P,gl,\gamma)$ be an $I_r$-colored shadow. Then the shadow state sum $|(P,gl)|_{\gamma}$ is a complex-valued regular isotopy invariant of colored shadows. Furthermore, this invariant gives rise to a complex-valued isotopy invariant of colored links in $S^1$-bundles over closed orientable surfaces.
\end{thm}

\begin{remark}
The notion of \textit{regular isotopy invariance of colored shadows} will not be relevant for our purposes since we only work with the projections arising from the construction of the links in Section \ref{AgolandLinksSec}. For details on the relationship between isotopy invariance of colored shadows and colored links in $S^1$-bundles over surfaces, we refer to \cite{TuraevShadow} Sections 2, 3, and 4.
\end{remark}

Turaev generalized the construction of colored shadows from the setting of colored links in $S^1$-bundles over closed surfaces \cite{TuraevShadow} to colored links in closed $3$-manifolds in \cite{TuraevBook}, Chapters IX and X. While this general characterization is not relevant to the arguments that follow, we include a brief description. Turaev shows that any $3$-manifold $N$ which is the boundary of a $4$-manifold $W$ can be associated a shadow $(P,gl)$. Roughly speaking, Turaev constructs the polyhedron $P$ using dual cell subdivisions of a triangulation of $N$ and equips gleams which encode the topology of the regular neighborhood of $P$ in $W$. The details of this construction can be found in IX.1 of \cite{TuraevBook}. Further, if $N$ contains a framed link $T$ colored by $\gamma$, Turaev extended this construction to a colored shadow $(P,gl,\gamma)$ associated to $(N,T,\gamma)$. In addition, we note this construction can be generalized to colored framed trivalent graphs contained in a $3$-manifold $N$. We refer the reader to X.7.1 of \cite{TuraevBook} for details.

For simplicity, we consider the following alternative construction, which can be found in \cite{CostantinoThurston}, to further support this more general notion of a shadow of a $3$-manifold. The framed link $L \subset S^3$ has a shadow $(P_0,gl_0)$ constructed by gluing a disk to $L \times [0,1]$ along $L \times \{0\}$. Surgery along a component of $L$ is equivalent to gluing the core of a $2$-handle to $P_0$, and this gluing does not change the gleam of the capped region. Due independently to Lickorish \cite{LickorishSurgery} and Wallace \cite{wallacesurgery}, any $3$-manifold $N$ can be obtained by performing integer surgery on a link in $S^3$, meaning that every $3$-manifold $N$ has a shadow with gleams related to its surgery presentation. We refer the reader to Chapter 12 of \cite{LickorishBook} and Chapter 9 of \cite{RolfsenBook} for more details on knot and link surgery.

\subsection{Relating the quantum invariants}\label{RelRTsubsec}

Here, we introduce the $r$-th relative Reshetikhin--Turaev invariants and their relation to Turaev's shadow state sum invariant established in \cite{TuraevBook}. We then define the Turaev--Viro invariants, which can be computed in terms of the $r$-th relative Reshetikhin--Turaev invariants \cite{growth6j, colJvolDKY}. We remark that the precise definitions presented here are not essential to understanding the arguments that follow, but nevertheless provide valuable context for the motivation behind our arguments.

We  give a brief outline of  the relative Reshetikhin--Turaev invariants using the skein theoretical approach. For more details, see \cite{BHMVKauffman} by Blanchet, Habegger, Masbaum, and Vogel as well as \cite{LickorishSkein} by Lickorish. Although we will define the relative Reshetikhin--Turaev invariants in terms of a surgery presentation for a $3$-manifold, we will see in Theorem \ref{RelRTshadowstatesum} that the invariants can also be constructed from the state sum invariants considered in Theorem \ref{Theorem:StateSum}, which will be more useful for our purposes. 

For an oriented $3$-manifold M and $r \geq 3$, we define the \emph{Kauffman bracket skein module} $K_r(M)$ of $M$ to be the $\mathbb{C}$-module generated by the isotopy classes of framed links in $M$ modulo the following relations: 
\begin{enumerate}[(I)]
    \item \emph{Kauffman bracket skein relation}: $\left \langle \KBCross  \right\rangle = q^{\frac{1}{2}} \left \langle \KBH  \right\rangle + q^{-\frac{1}{2}} \left \langle \KBV  \right\rangle $.
    
    \item \emph{Framing relation}: $\left \langle L \sqcup \KBCirc  \right\rangle = \left( -q - q^{-1} \right) \left \langle L \right \rangle$.
\end{enumerate}
In the case when $M=S^3$, the Kauffman bracket skein module $K_r(S^3)$ is $1$-dimensional, and we obtain an isomorphism 
$$\left \langle \cdot \right \rangle: K_r(S^3) \to \mathbb{C}$$
by sending the empty diagram to $1$. For a link $L \subset S^3$, we call the image 
$\left \langle L \right \rangle \in \mathbb{C}$ the \emph{Kauffman bracket} of $L$.

We now consider the Kauffman bracket skein module $K_r(S^1\times [0,1]^2)$  of the solid torus. For any framed link $L \subset S^3$ with $k$ ordered components and $b_1, b_2, \ldots,b_k \in K_r(S^1\times[0,1]^2)$, we define the  $\mathbb{C}$-multilinear map 
\[
\left \langle \cdot, \dots, \cdot \right \rangle_L: K_r(S^1 \times [0,1]^2) \to \mathbb{C},
\]
where $\left \langle b_1, b_2, \dots, b_k \right \rangle_L$ is the
 cabling of the components of $L$ by $b_1, b_2,\dots, b_k$ followed by evaluating in $K_r(S^3)$ using the Kauffman bracket. 
 
 On $K_r(S^1\times[0,1]^2)$, there is a commutative multiplication induced by juxtaposition of annuli $S^1 \times [0,1] \times \{pt\}$ making $K_r(S^1\times [0,1]^2)$ a $\mathbb{C}$-algebra. By sending the core of the annuli to the indeterminate $z$, we obtain the isomorphism $K_r(S^1\times [0,1]^2) \cong \mathbb{C}[z]$. We will not go into detail; however, we can construct  specific elements $e_m, \omega_r \in K_r(S^1 \times [0,1]^2)$ where $m \in I_r$. The  $e_m$ will correspond to colorings of our link, and the element $\omega_r$ is known as the \emph{Kirby coloring} which will allow us to define an invariant for a framed link in any closed oriented $3$-manifold. We will now define the $r$-th relative Reshetikhin--Turaev invariants. 
 \begin{definition}
 Let $M$ be a closed oriented $3$-manifold presented in $S^3$ by surgery along the framed link $L'$ with $n'$ components, and let $L$ be a framed link in $S^3$ with $n$ components. We consider the link $L \sqcup L' \subset S^3$ with $n+n'$ components where the first $n$ components correspond to the components of  $L$. For a coloring $\gamma = (\gamma_1, \gamma_2, \dots, \gamma_n) \in I_r^n$ of components of $L$, we define the \emph{r-th relative Reshetikhin--Turaev invariants} as 
 $$RT_r(M,L,\gamma) = \mu_r \left \langle e_{\gamma_1}, \dots, e_{\gamma_n}, \omega_r, \dots, \omega_r     \right \rangle_{L\sqcup L'} \left \langle \omega_r \right \rangle_{U_+}^{-\sigma(L')}$$
 where $\mu_r$ is a constant dependent on $r$, $U_+$ is the $+1$ framed unknot, and $\sigma(L')$ is the signature of the linking matrix of $L'$. For the explicit constructions of $\mu_r$, $e_m$, and $\omega_r$, we again reference \cite{BHMVKauffman, LickorishSkein}. 
  \end{definition}

Turaev's constructions in IX and X of \cite{TuraevBook} establish a deep relationship between the $r$-th relative Reshetikhin--Turaev invariants and the generalized shadow state sum invariant. Notably, one can study the $r$-th relative Reshetikhin--Turaev invariants of a colored framed link in a closed $3$-manifold from the perspective of colored shadows. In \cite{costantinoColoredJones}, Costantino used this relationship to study the colored Jones invariants of links in $S^3 \#_k \left(S^2 \times S^1\right)$ from the shadow perspective. We include Costantino's statement of Turaev's result from X.7.1 of \cite{TuraevBook} here.

\begin{thm}[\cite{costantinoColoredJones}, Theorem 3.3]\label{RelRTshadowstatesum}
Let $N$ be a closed $3$-manifold and $T\subset N$ a colored framed trivalent graph in $N$ colored by $\gamma$. Let $(P,gl,\gamma)$ be a colored shadow of $(N,T)$. Then $RT_r(N,T,\gamma):= C_{r} |(P,gl)|_{\gamma}$ is a complex-valued homeomorphism invariant of $(N,T)$.
\end{thm}

\begin{remark}
Theorem \ref{RelRTshadowstatesum} is stated generally in terms of framed trivalent graphs rather than framed links to be consistent with the literature \cite{costantinoColoredJones, TuraevBook}. The full generality of the result is not necessary for the arguments that follow since we only consider manifolds containing framed links.
\end{remark}

\begin{remark}
Here, the factor $C_r$ is considered a ``normalization factor." See \cite{costantinoColoredJones} for a precise formulation. In the case Turaev \cite{TuraevShadow} studies, where $N$ is homeomorphic to an $S^1$-bundle over a closed surface and $T$ is a link, the factor $C_{r}$ does not depend on $T$. It can therefore be ignored for our purposes.
\end{remark}

Finally, we introduce the $r$-th Turaev--Viro invariant and a method for calculating it in terms of the $r$-th relative Reshetikhin--Turaev invariants \cite{growth6j, colJvolDKY}. Turaev and Viro \cite{TuraevViro} defined a real-valued topological invariant on a triangulation of a compact $3$-manifold for fixed $r$ and a root of unity $q$ using quantum $6j$-symbols. We will define the $SU(2)$-version of the invariant for odd $r \geq 3$ and $q = e^{\frac{2\pi \sqrt{-1}}{r}}$ in terms of the quantum $6j$-symbols defined in Subsection \ref{Q6jSubSec}, following the conventions of \cite{growth6j}. We begin by introducing the notion of an admissible coloring of a triangulation.

\begin{definition}
An \textit{$r$-admissible coloring} of a tetrahedron $T$ is a map assigning an $r$-admissible 6-tuple $(a_1,\dots,a_6) \in I_r^6$ to the edges of $T$. In particular, the triples $(a_i,a_j,a_k)$ corresponding to each face of $T$ must be $r$-admissible triples satisfying Definition \ref{rAdmissible}. We say that a coloring of the edges of a triangulation $\tau$ of a $3$-manifold is \textit{$r$-admissible} if each tetrahedron $T \in \tau$ admits an $r$-admissible coloring.
\end{definition}

\begin{definition}
Let $M$ be a compact orientable $3$-manifold with boundary $\partial M$. We say $\tau$ is a \textit{partially ideal triangulation} of $M$ if some vertices of the triangulation are truncated, and the faces of the truncated vertices form a triangulation of $\partial M$.
\end{definition}

Let $Adm(r, \tau)$ denote the set of $r$-admissible colorings of $\tau$. Denote the set of interior vertices of $\tau$ by $V$ and the set of interior edges of $\tau$ by $E$. Given a coloring $\gamma \in Adm(r,\tau)$, and edge $e\in E$, and a tetrahedron $T\in \tau$, we define
\[
|e|_{\gamma} := (-1)^{\gamma(e)}[\gamma(e)+1],
\]
and $|T|_{\gamma}$ to be the quantum $6j$-symbol associated to the 6-tuple assigned to $T$ by $\gamma$.

\begin{definition}\label{TVDefinition}
Fix $r \geq 3$ be odd and $q = e^{\frac{2\pi \sqrt{-1}}{r}}$. Let $M$ be a compact orientable $3$-manifold with boundary $\partial M$, and let $\tau$ be a partially ideal triangulation of $M$. Then the $r$-th Turaev--Viro invariant at the root $q$ is given by
\begin{align}
    TV_r(M,\tau;q) &:= \left(\frac{\sqrt{2}\sin\left(\frac{2\pi}{r}\right)}{\sqrt{r}}\right)^{2|V|} \sum_{\gamma \in Adm(r,\tau)} \left( \prod_{e \in E} |e|_{\gamma} \prod_{T\in \tau} |T|_{\gamma} \right).
\end{align}
\end{definition}

The quantity $TV_r(M,\tau;q)$ is independent of the choice of partially ideal triangulation $\tau$ by \cite{TuraevViro}, so it is a topological invariant of $M$. For a link complement $M\backslash L$, this means $TV_r(M\backslash L;q)$ can be computed by summing over $I_r$-colorings of the link $L \subset M$. In \cite{growth6j} and \cite{colJvolDKY}, the authors prove that the $r$-th Turaev--Viro invariant of a link complement $M\backslash L$ can be computed via the $r$-th relative Reshetikhin--Turaev invariants of $(M,L)$.
\begin{prop}[\cite{growth6j,colJvolDKY}]\label{TVRTProp}
Let $M$ be a $3$-manifold, $L \subset M$ be a link with $k$ components, and  $RT_r(M,L,\gamma)$ be the $r$-th relative Reshetikhin--Turaev invariant of $(M,L)$ with the link $L$ colored by $\gamma \in I_r^k $. Then the Turaev--Viro invariant of the complement $M \backslash L$ at $q = e^{\frac{2\pi \sqrt{-1}}{r}}$ is given by
\begin{align}\label{TVRTSum}
TV_r(M\backslash L;q) = \sum_{\gamma \in I_r^k} \left|RT_r(M,L,\gamma)\right|^2.
\end{align}
\end{prop}

%%%%%%%%%%%%%%%%%%%%%%%%%%%%%%%%%%%%%%%%%%%%%%%%%%%%%%%%%%
%%% SECTION 4
%%%%%%%%%%%%%%%%%%%%%%%%%%%%%%%%%%%%%%%%%%%%%%%%%%%%%%%%%%

\section{Proof of Theorem \ref{mainthm}}\label{mainThmProofSec}
In this section, we prove Theorem \ref{mainthm} which we restate here.

\begin{thm}\label{mainthmrestate}
Let $M_L(k,l) \in \mathcal{M}$. Then for $r$ running over odd integers and $q = e^{\frac{2\pi \sqrt{-1}}{r}}$,
\begin{align*}
\lim_{r \rightarrow \infty} \frac{2\pi}{r} \log |TV_r(M_L(k,l);q)| = v_3 ||M_L(k,l)|| = 2(k+2l)v_8
\end{align*}
where $v_8 \approx 3.66$ is the volume of the regular ideal hyperbolic octahedron.
\end{thm}

To do this, we first write a formula for the shadow state sum invariants $|(P,gl)|_{\gamma}$ for the family $\mathcal{M}$ of links in trivial $S^1$-bundles over surfaces constructed in Subsection \ref{linkfamilySubSec}. We will then state and prove Lemma \ref{techlem} regarding the asymptotics of $|(P,gl)|_{\gamma}$. We complete the proof of Theorem \ref{mainthm} using Lemma \ref{techlem} and the formulation of the Turaev--Viro invariants from Subsection \ref{RelRTsubsec} in terms of the $r$-th relative Reshetikhin--Turaev invariants.

\begin{remark}
Computing the relative Reshetikhin--Turaev invariants is difficult in general, but the family $\mathcal{M}$ was constructed in order to simplify their calculation significantly from the shadow state sum perspective. In particular, shadows of these manifolds have simple gleams and topologically simple regions that allow us to reduce the proof of Theorem \ref{mainthmrestate} to studying properties of quantum $6j$-symbols. These manifolds also have well-understood simplicial volumes determined by $k$ and $l$.
\end{remark}

Let $M_L(k,l) \in \mathcal{M}$ be the complement of a link $L$ in a $3$-manifold $M = \Sigma_{g} \times S^1$ constructed as in Subsection \ref{linkfamilySubSec}. By Theorem \ref{canonShadow}, $M_L(k,l)$ has a shadow $(P,gl)$ associated to it. $M_L(k,l)$ has an elementary decomposition into $k$ $S-$pieces and $l$ $A-$pieces. The shadow $(P,gl)$ has a corresponding decomposition, so we also refer to these shadows on $\Sigma_{1,1}$ and $\Sigma_{0,4}$ as $S-$pieces and $A-$pieces, respectively. An $S-$piece has two loops which intersect at a single vertex. Let $s_1,\dots,s_k$ denote the intersection points on the $k$ $S-$pieces. An $A-$piece has two loops which intersect at two vertices. Let $(a_1^1,a_1^2),\dots, (a_{l}^1,a_{l}^2)$ denote the intersection points on the $l$ $A-$pieces. 

We make the following observations about $M_L(k,l)$:
\begin{itemize}
\item Each $S-$piece of $(P,gl)$ has two curves, one vertex, and one region $X$ with $z = 4$ corners as in Figure \ref{fig:SShadow}. This region has gleam $x=2$ since gleams are assigned to regions for trivial bundles as in Figure \ref{fig: GleamTrivial}. Cutting $\Sigma_{1,1}$ along one of the two curves produces a pair of pants such that the second curve becomes a simple arc connecting the two new boundary components. Cutting along this arc produces an annulus, so $X$ has Euler characteristic $\chi =0$.  The modified gleam of the region of this shadow is $x' = x - z/2 = 0$.

\item Each $A-$piece of $(P,gl)$ has two curves, two vertices, and four regions $X_t$, $t=1,2,3,4$, each with $z_t = 2$ corners, as in Figure \ref{fig:AShadow}. Again, using the gleam assignment from Figure \ref{fig: GleamTrivial}, each of the four regions has gleam $x_t=1$. Cutting $\Sigma_{0,4}$ along one of the two curves separates $\Sigma_{0,4}$ into two pairs of pants such that the second curve is split into a simple arc on each pair of pants with endpoints on a single boundary component. Cutting the two pairs of pants along these arcs produces four annuli, so $X_t$ has Euler characteristic $\chi_t =0$ for $t=1,2,3,4$. The modified gleam of each region of this shadow is $x_t' = x_t - z_t/2 = 0$.
\end{itemize}

\begin{figure}[!htb]
    \centering
    \begin{subfigure}[b]{0.3\textwidth}
        \centering
	    \includegraphics[width=0.6\textwidth]{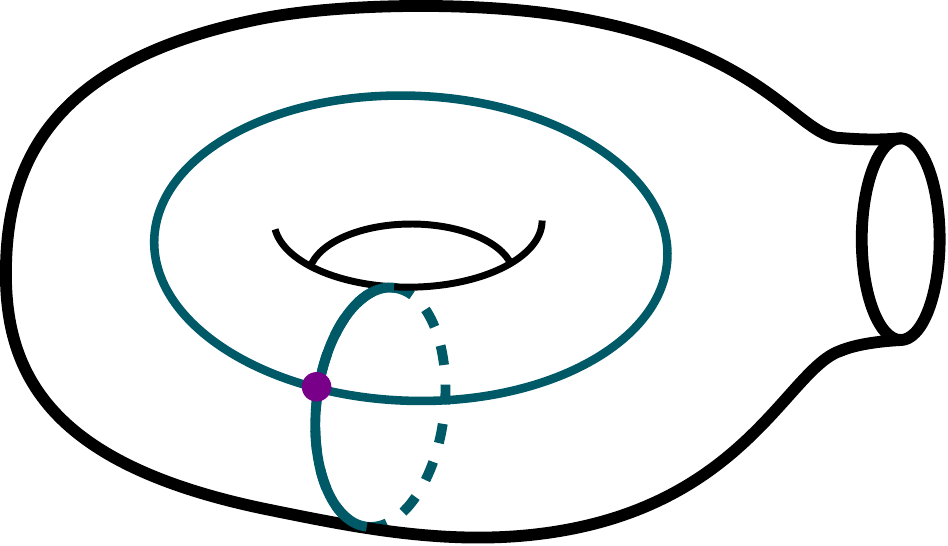}
	    \caption{Shadow corresponding to a pair of $S-$moves. The single region has gleam 2.}
	    \label{fig:SShadow}
	\end{subfigure}
	\hspace{1cm}
	\begin{subfigure}[b]{0.3\textwidth}
	    \centering
	    \includegraphics[width=0.6\textwidth]{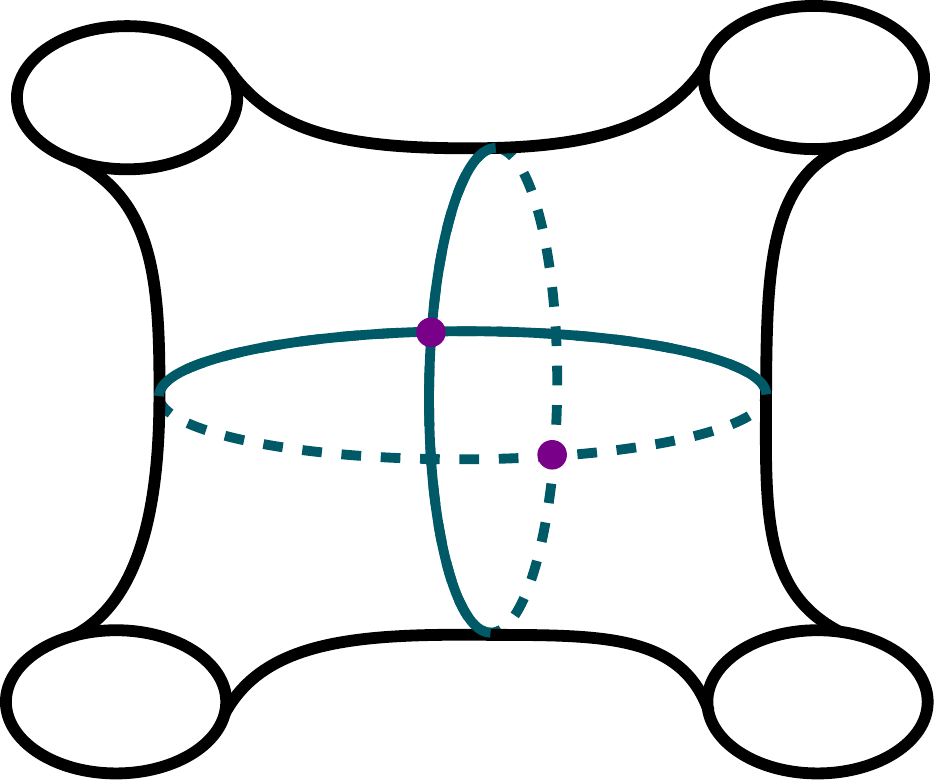}
	    \caption{Shadow corresponding to a pair of $A-$moves. Each region has gleam 1.}
	    \label{fig:AShadow}
	\end{subfigure}
	\caption{$S-$ and $A-$pieces.}
	\label{fig:smallShadows}
\end{figure}

\begin{remark}
While the underlying surface of the shadow $(P,gl)$ associated to $M_L(k,l)$ is closed, the underlying surfaces of the $S-$ and $A-$pieces have boundary. By the construction of $M_L(k,l)$, the annular regions of the $S-$ and $A-$pieces are glued along their boundaries to form the regions of $(P,gl)$. The gleam (resp. Euler characteristic) of a region $X$ of $(P,gl)$ is the sum of the gleams (resp. Euler characteristic) of the regions in the $S-$ and $A-$pieces that are glued to form $X$.
\end{remark}

Let $\gamma \in I_r^{2k+2l}$ be an $I_r$-coloring of $L \subset M$. The loops of the associated shadow $(P,gl)$ inherit this $I_r$-coloring $\gamma$. Let $\eta \in \text{adm}\left(P,gl,\gamma\right)$ be an admissible surface-coloring of $\left(P,gl,\gamma\right)$. The $S-$ and $A-$piece observations imply that for the state $|(P,gl)|_{\gamma}^{\eta}$ defined in Equation (\ref{state}), $\chi_t = x_t' = 0$ for all regions $X_t$. This means
\begin{align*}
\prod_{t=1}^q  \left(\left(v_{\eta(X_t)}\right)^{\chi_t} \text{exp}\left(2 u_{\eta(X_t)} x_t'\right)  \right) &= 1.
\end{align*} 
Using this and Equation (\ref{state}), we reformulate the state sum invariant of the $I_r$-colored shadow associated to $M_L(k,l)$ in the following proposition.

\begin{prop}\label{RelRTFormula}
Let $(P,gl,\gamma)$ be a colored shadow associated to $M_L(k,l)$. Then the shadow state sum invariant is given by 
\begin{align}
|(P,gl)|_{\gamma} &= \sum_{\eta \in \text{adm}(P,gl,\gamma)} |(P,gl)|_{\gamma}^{\eta} \nonumber\\
	&= \sum_{\eta \in \text{adm}(P,gl,\gamma)} \prod_{i = 1}^k |s_i|^{\eta}\prod_{j=1}^l |a_j^1|^{\eta}|a_j^2|^{\eta}. \label{sumof6j}
\end{align}
\end{prop}

Since Equation (\ref{sumof6j}) is a sum of products of quantum $6j$-symbols, Proposition \ref{RelRTFormula} allows us to use the explicit properties of quantum $6j$-symbols discussed in Subsection \ref{Q6jSubSec}. The following technical lemma will be used to prove Theorem \ref{mainthmrestate}.

\begin{remark}
We will use the abbreviation $(n):= (n,\dots, n)$ for tuples of colors throughout the rest of the paper.
\end{remark}

\begin{lem}\label{techlem}
Let $M_L(k,l) \in \mathcal{M}$ and $\gamma = (n_r) \in I_r^{2k+2l}$, where $n_r:= \frac{r - 1}{2}$ when $r \equiv 1 \mod 4$ and $n_r:= \frac{r -3}{2}$ when $r \equiv 3 \mod 4$. Let $(P,gl,\gamma)$ be the $I_r$-colored shadow  representing $M_L(k,l)$. Then  
\begin{align}\label{lemRTlim}
\lim_{r \rightarrow \infty} \frac{4\pi}{r} \log \left||(P,gl)|_{\left(n_r\right)}\right| = v_3 \|M_L(k,l)\|,
\end{align}
where $v_3\|M_L(k,l)\| = 2(k+2l)v_8$ is the simplicial volume of $M_L(k,l)$.
\end{lem}

The following two lemmas give properties of quantum $6j$-symbols that will be leveraged to establish a lower bound for the limit in Equation (\ref{lemRTlim}) of Lemma \ref{techlem}.

\begin{lem}\label{Real6jlem}
Let $n_r:= \frac{r - 1}{2}$ when $r \equiv 1 \mod 4$ and $n_r:= \frac{r -3}{2}$ when $r \equiv 3 \mod 4$. Suppose $m_1,m_2,m_3,m_4 \in I_r$ and the tuple $(n_r, m_1,m_2,n_r, m_3,m_4)$ is $r$-admissible. Then the quantum $6j$-symbol
\begin{align*}
    \left|\begin{array}{ccc}
    n_r & m_1 & m_2 \\
    n_r & m_3 & m_4
    \end{array} \right|
\end{align*}
is real-valued.
\end{lem}

\begin{proof}
Let $n_r:= \frac{r - 1}{2}$ when $r \equiv 1 \mod 4$ and $n_r:= \frac{r -3}{2}$ when $r \equiv 3 \mod 4$. Note that $n_r$ is always even. Consider the quantum $6j$-symbol associated to the tuple $(n_r, m_1,m_2,n_r, m_3,m_4)$. From Definition \ref{q6jdef}, the quantum $6j$-symbol associated to the 6-tuple $(a_1,\dots, a_6)$ is either real or purely imaginary  based on the value of the coefficient 
\begin{align}\label{complexfactors}
\sqrt{-1}^{\left(-\sum_{i=1}^6 a_i\right)} \Delta(a_1,a_2,a_3)\Delta(a_1,a_5,a_6)
    \Delta(a_2,a_4,a_6)\Delta(a_3,a_4,a_5),
\end{align}
since the sum in Equation (\ref{q6j}) is real-valued.

For $(n_r, m_1,m_2,n_r, m_3,m_4)$, the first factor is given by
\begin{align*}
(\sqrt{-1})^{-\left(n_r + m_1+m_2+n_r+ m_3+m_4\right)} &= (\sqrt{-1})^{m_1+m_2+ m_3+m_4} = \pm 1.
\end{align*}
The first equality  holds because $n_r$ is even, so $2n_r$ has a factor of $4$. The second equality is due to the admissibility conditions which require that each of the sums $m_1+m_2$, $m_3+m_4$, $m_1+m_4$, and $m_2+m_3$ are even. Notice in the case of $(n_r, m,m,n_r, m,m)$, the factor becomes
\begin{align*}
    (\sqrt{-1})^{-\left(2n_r + 4m\right)} &= 1.
\end{align*}

For the other factors of Equation (\ref{complexfactors}), suppose $(n_r,m,m')$ is an $r$-admissible triple and without loss of generality, assume $m\geq m'$. Then 
\[
\Delta(n_r,m,m') = \sqrt{\frac{\left[\frac{n_r+m-m'}{2}\right]! \left[\frac{m'+n_r-m}{2}\right]! \left[\frac{m+m'-n_r}{2}\right]!}{\left[\frac{n_r+m+m'}{2} +1 \right]!}}.
\]
By the admissibility conditions, $\frac{n_r+m-m'}{2} \leq n_r < \frac{r}{2}$, $\frac{n_r+m'-m}{2} \leq \frac{n_r}{2} < \frac{r}{2}$, and $\frac{m+m'-n_r}{2} \leq r-2-n_r < \frac{r}{2}$. Since $[n]>0$ for $0\leq n <\frac{r}{2}$, the numerator of $\Delta(n_r,m,m')$ is real-valued. This implies the numerator of $\displaystyle \Delta(n_r,m_1,m_2)\Delta(n_r,m_3,m_4)\Delta(m_1,n_r,m_4)\Delta(m_2,n_r,m_3)$ is also real-valued.
In addition, the admissibility conditions imply that $\frac{r-1}{2} \leq n_r+1 \leq \frac{n_r+m+m'}{2}+1 \leq r-1$, so the sign of $\left[\frac{n_r+m+m'}{2}+1 \right]!$
is given by 
\[
\left( -1\right)^{\frac{n_r+m+m'}{2}+1-\frac{r-1}{2}}.
\]
This means the denominator of $\Delta(n_r,m_1,m_2)\Delta(n_r,m_3,m_4)\Delta(m_1,n_r,m_4)\Delta(m_2,n_r,m_3)$ is some real-valued multiple of
\begin{align*}
& \sqrt{\left( -1\right)^{2n_r+4-2(r-1)+m_1+m_2+m_3+m_4}} = \pm 1,
\end{align*}
where equality holds because $2n_r+4-2(r-1)$ contains a factor of 4 and $m_1+m_2+m_3+m_4$ is even. Hence, the coefficient given by Equation (\ref{complexfactors}) is real-valued. This implies that the quantum $6j$-symbol associated to the $6$-tuple $(n_r, m_1,m_2,n_r, m_3,m_4)$ is real-valued. Notice in the case of $(n_r, m,m,n_r, m,m)$, the coefficient given by Equation (\ref{complexfactors}) is positive.
\end{proof}

\begin{lem}\label{Allm6jsignlem}
Let $n_r:= \frac{r - 1}{2}$ when $r \equiv 1 \mod 4$ and $n_r:= \frac{r -3}{2}$ when $r \equiv 3 \mod 4$. Let $m \in I_r$ and suppose that the tuple $(n_r, m,m,n_r, m,m)$ is $r$-admissible. Then the sign of
\begin{align*}
    \left|\begin{array}{ccc}
    n_r & m & m \\
    n_r & m & m
    \end{array} \right|
\end{align*}
is independent of $m$. Moreover, it is positive when $r \equiv 3 \mod 4$ and negative when $r \equiv 1 \mod 4$.
\end{lem}

\begin{proof}
Consider the quantum $6j$-symbol associated to the 6-tuple $(n_r, m,m,n_r, m,m)$. By Lemma \ref{Real6jlem}, this quantum $6j$-symbol is real-valued. The admissibility conditions imply that $\frac{n_r}{2} \leq m \leq r-2- \frac{n_r}{2}$. By Definition \ref{q6jdef},
\begin{align}\label{Allm6j}
	\left| \begin{array}{ccc}	
	n_r & m & m \\
    n_r & m & m
    \end{array} \right| &=
    \Delta\left(n_r,m,m\right)^4 
    \left(\sum_{k = m+\frac{n_r}{2}}^{\text{min}\{m+n_r,2m\}} S_{m,k}\right)
\end{align}
where 
\begin{align*}
S_{m,k} = \frac{(-1)^k [k+1]!}
{\left[k-\left(m+ \frac{n_r}{2}\right)\right]!^4 \left[m+n_r-k\right]!^2\left[2m-k\right]!}.
\end{align*}
Suppose that $r \equiv 1 \mod 4$. By the admissibility conditions, the 6-tuple $(\frac{r-1}{2}, m,m,\frac{r-1}{2}, m,m)$ satisfies assumptions $(a)$ and $(b)$ of Theorem \ref{growthtrunctetra}. In the case that $r \equiv 3 \mod 4$, the admissibility conditions of the 6-tuple $(\frac{r-3}{2}, m,m,\frac{r-3}{2}, m,m)$ imply it satisfies  assumptions $(a)$ and $(b)$ of Theorem \ref{growthtrunctetra} for all admissible region colors except $m = \frac{n_r}{2} = \frac{r-3}{4}$ and $m = r-2-\frac{n_r}{2} =\frac{3r-5}{4}$. We will consider these cases separately.

\vspace{0.1cm}
\textbf{General Case:}

Suppose either $r \equiv 1 \mod 4$ or $r\equiv 3 \mod 4$ with $\frac{r-3}{4} <m< \frac{3r-5}{4}$. Then by part $(1)$ of Theorem \ref{growthtrunctetra}, the sign of $S_{m,k}$ is independent of $k$. We now show that the sign of $S_{m,k}$ is independent of the region color $m$. Without loss of generality, consider the case $k= m+\frac{n_r}{2}$:
\[
S_{m, m+\frac{n_r}{2}} = \frac{(-1)^{m+\frac{n_r}{2}} [m+\frac{n_r}{2}+1]!}
{\left[\frac{n_r}{2}\right]!^2\left[m-\frac{n_r}{2}\right]!}.
\]
Since the quantum integer $[n]$ is real-valued, we only need to consider the signs of $[m+\frac{n_r}{2}+1]!$ and $\left[m-\frac{n_r}{2}\right]!$. By assumption $(a)$ of Theorem \ref{growthtrunctetra} and the assumption that $m<r-2-\frac{n_r}{2}$, we know $0\leq m-\frac{n_r}{2}\leq \frac{r-2}{2}$, so $\left[m-\frac{n_r}{2}\right]! >0$ for all region colors $m$. By assumption $(b)$ of Theorem \ref{growthtrunctetra} and the assumption that $m>\frac{n_r}{2}$, we know $\frac{r-2}{2} \leq m+\frac{n_r}{2} \leq
r-2$, so $[m+\frac{n_r}{2}+1] <0$ for all region colors $m$. Note that $\left[\frac{r - 1}{2}\right]! > 0$, so $[m+\frac{n_r}{2}+1]! = [m+\frac{n_r}{2}+1] \cdots [\frac{r + 1}{2}][\frac{r - 1}{2}]!$ has sign
\[
(-1)^{m+\frac{n_r}{2}+1-\frac{r-1}{2}}.
\]
Then the sign of $S_{m, m+\frac{n_r}{2}}$ is 
\begin{align}\label{signofSk}
(-1)^{m+\frac{n_r}{2}+m+\frac{n_r}{2}+1-\frac{r-1}{2}} = (-1)^{-\frac{r-3}{2}},
\end{align}
which is independent of the region color $m$. By part $(1)$ of Theorem \ref{growthtrunctetra} and Equation (\ref{Allm6j}), the sign of the quantum $6j$-symbol associated to $(n_r, m,m,n_r, m,m)$ is independent of the region color $m$, provided $m \neq \frac{r-3}{4}, \frac{3r-5}{4}$ in the case $r\equiv 3 \mod 4$. 

\vspace{0.1cm}
\textbf{Special Cases:}

If $m = \frac{r-3}{4}$, we have $\max{T_i} = \frac{r-3}{2} = \min{Q_j}$, so 
\begin{align*}
	\left| \begin{array}{ccc}	
	\frac{r-3}{2} & \frac{r-3}{4} & \frac{r-3}{4} \\
    \frac{r-3}{2} & \frac{r-3}{4} & \frac{r-3}{4}
    \end{array} \right| &=
    \Delta\left(\frac{r-3}{2},\frac{r-3}{4},\frac{r-3}{4}\right)^4 
   \left(\frac{(-1)^{\frac{r-3}{2}}\left[\frac{r-3}{2}+1\right]!}{\left[ \frac{r-3}{4} \right]!^2}\right) >0.
\end{align*}
Positivity follows because $n_r = \frac{r-3}{2}$ is even and $0<\frac{r-3}{2}+1<\frac{r}{2}$.

\vspace{0.1cm}
If $m = \frac{3r-5}{4}$, we have $\max{T_i} = r-2$ and $\min{Q_j} = r-2 + \frac{r-3}{4}$. However, since $[k+1]! = 0$ for $k>r-2$,
\begin{align*}
	\left| \begin{array}{ccc}	
	\frac{r-3}{2} & \frac{3r-5}{4} & \frac{3r-5}{4} \\
    \frac{r-3}{2} & \frac{3r-5}{4} & \frac{3r-5}{4}
    \end{array} \right| &=
    \Delta\left(\frac{r-3}{2},\frac{3r-5}{4},\frac{3r-5}{4}\right)^4 
   \left(\frac{(-1)^{r-2}\left[r-1\right]!}{\left[ \frac{r-3}{4} \right]!^2 \left[\frac{r-1}{2} \right]!}\right) >0.
\end{align*}
Positivity follows because the denominator is positive and the numerator is given by $(-1)^{r-2} [r-1][r]\cdots \left[\frac{r+1}{2}\right]\left[\frac{r-1}{2}\right]!$, which has sign 
\[
(-1)^{r-2+\left(r-1-\frac{r-1}{2}\right)} = (-1)^{\frac{3r-5}{2}} = 1.
\]
Thus the sign of the quantum $6j$-symbol associated to $(n_r, m,m,n_r, m,m)$ is independent of $m$. Moreover, by Equation (\ref{signofSk}), this $6j$-symbol is negative when $r \equiv 1 \mod 4$ and positive when $r\equiv 3 \mod 4$.
\end{proof}

We now prove Lemma \ref{techlem}. 
\begin{proof}[Proof of Lemma \ref{techlem}]
From Subsection \ref{linkfamilySubSec}, the link complement $M_L(k,l)$ is a compact orientable $3$-manifold with simplicial volume $v_3\|M_L(k,l)\| = 2(k+2l)v_8$. We proceed by bounding the limit in Equation (\ref{lemRTlim}) above and below by $v_3\|M_L(k,l)\|$. 

\vspace{0.2cm}
\textbf{\underline{Step 1: The upper bound}}

For the upper bound, note that each summand in Equation (\ref{sumof6j}) is a product of $k+2l$ quantum $6j$-symbols. By Theorem \ref{growthrate6jThm}, the growth rate of a single summand of Equation (\ref{sumof6j}) is bounded above sharply by $(k+2l)v_8$. Let $B_r = \#adm(P,gl,\gamma)$ be the number of $r$-admissible surface-colorings of $(P,gl,\gamma)$. The term $B_r$ grows at most polynomially with $r$ since $B_r$ is bounded above by the total number of $r$-admissible 6-tuples corresponding to well-defined quantum $6j$-symbols. Thus, we obtain the following upper bound:
\begin{align*}
\limsup_{r \rightarrow \infty} \frac{4\pi}{r} \log \left||(P,gl)|_{\gamma}\right| & \leq
    \limsup_{r \rightarrow \infty} \frac{4\pi}{r} \log \left|B_r \max_{\eta \in adm(P,gl,\gamma)}|(P,gl)|_{\gamma}^{\eta}\right| \\
    & = 2 \limsup_{r \rightarrow \infty} \frac{2\pi}{r} \log \left| \max_{\eta \in adm(P,gl,\gamma)}|(P,gl)|_{\gamma}^{\eta}\right| \\
    &\leq 2(k+2l)v_8,
\end{align*}
where the last inequality is due to Theorem \ref{growthrate6jThm}.

We remark that this upper bound holds for any $I_r$-coloring $\gamma$, which implies the required upper bound for the limit in Equation (\ref{lemRTlim}). 

\vspace{0.2cm}
\textbf{\underline{Step 2: The lower bound}}

Let $n_r:= \frac{r - 1}{2}$ when $r \equiv 1 \mod 4$ and $n_r:= \frac{r -3}{2}$ when $r \equiv 3 \mod 4$. We will prove the lower bound for the $\left(n_r\right)$-colored link. 

For the lower bound, it suffices to show that summands of Equation (\ref{sumof6j}) do not cancel with each other when $\gamma= \left(n_r \right)$. In particular, we will show that, for fixed $r$, the sign of every summand of Equation (\ref{sumof6j}) is independent of the surface-coloring $\eta \in \text{adm}\left(P,gl,\left(n_r \right)\right)$. This means that the absolute value of any individual summand is a lower bound for $\left||(P,gl)|_{\left(n_r\right)}\right|$. We now make some observations about $ \prod_{i = 1}^k |s_i|^{\eta}\prod_{j=1}^l |a_j^1|^{\eta}|a_j^2|^{\eta}$.
\begin{itemize}
\item The surface-coloring of each $S-$piece is given by Figure \ref{fig: SShadowAreaColored}. This means each factor $|s_i|^{\eta}$ is the quantum $6j$-symbol associated to the 6-tuple $(n_r, m,m,n_r, m,m)$.

\item The surface-coloring of each $A-$piece is given by Figure \ref{fig: AShadowAreaColored}. Using the symmetries of the quantum $6j$-symbol in Equation (\ref{q6jSymmetries}), each factor $|a_j^i|^{\eta}$, for $i=1,2$, is the quantum $6j$-symbol associated to the 6-tuple $(n_r, m_1,m_2,n_r, m_3,m_4)$.
\end{itemize}

\begin{figure}[!htb]
    \centering
    \begin{subfigure}[b]{0.3\textwidth}
        \centering
	    \includegraphics[width=0.6\textwidth]{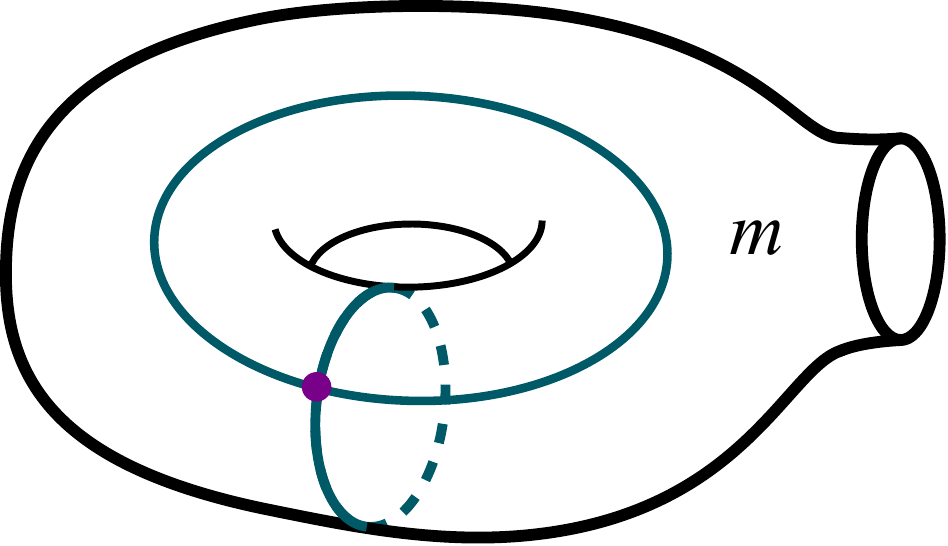}
	    \caption{Surface-coloring of a shadow corresponding to a pair of $S-$moves.}
	    \label{fig: SShadowAreaColored}
	\end{subfigure}
	\hspace{1cm}
	\begin{subfigure}[b]{0.3\textwidth}
	    \centering
	    \includegraphics[width=0.6\textwidth]{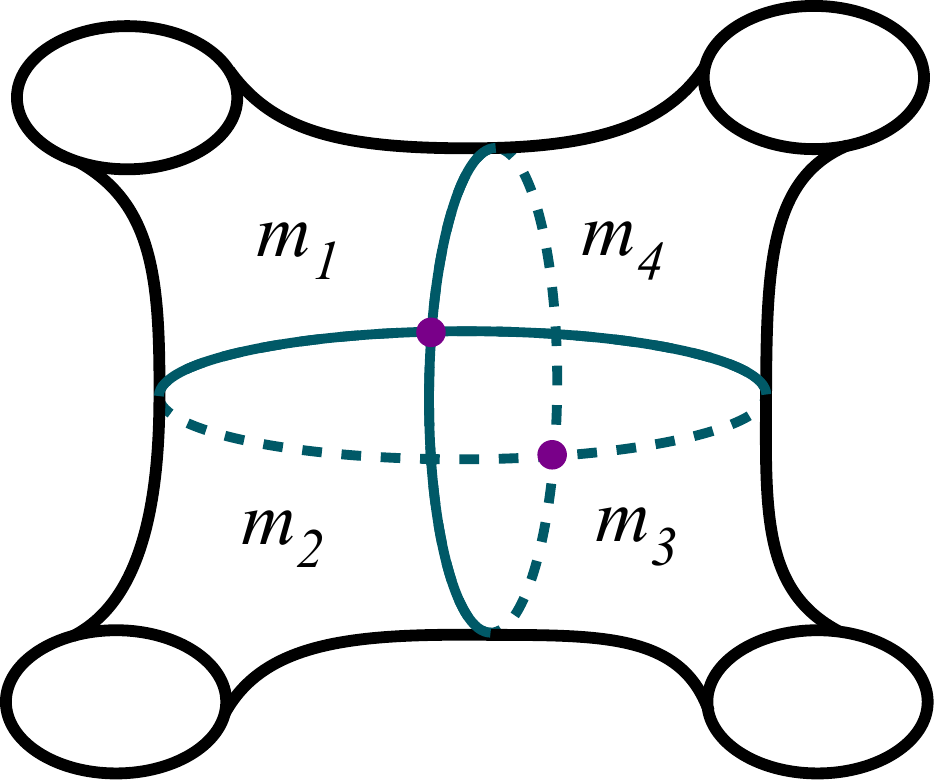}
	    \caption{Surface-coloring of a shadow corresponding to a pair of $A-$moves.}
	    \label{fig: AShadowAreaColored}
	\end{subfigure}
	\caption{Surface-colored shadows.}
	\label{fig:smallShadowsAreaColored}
\end{figure}

Using these observations, we can re-formulate Equation (\ref{sumof6j}) for the $\left(n_r\right)$-colored shadow:
\begin{align}\label{sumof6jspecific}
|(P,gl)|_{\left(n_r\right)} &= \sum_{\eta \in \text{adm}\left(P,gl,\left(n_r \right)\right)} 
	\prod_{i = 1}^k \left| \begin{array}{ccc}
    n_r & m^i & m^i \\
    n_r & m^i & m^i
    \end{array} \right| 
    \prod_{j=1}^l \left| \begin{array}{ccc}
    n_r & m_1^j & m_2^j \\
    n_r & m_3^j & m_4^j
    \end{array} \right|^2
\end{align}

We remark that the notation of Equation (\ref{sumof6jspecific}) is chosen out of convenience. The construction of the invariant may introduce dependencies between surface-colors and hence between entries of the quantum $6j$-symbols. For example, if an $S-$piece with region colored by $m^i$ is glued to an $A-$piece along the boundary circle adjacent to a region colored by $m^j$, the regions combine to form a single region with color $m^i = m^j$.
We choose to omit these additional details since they do not change the overall result of Lemma \ref{techlem}.

By Lemma \ref{Real6jlem}, the quantum $6j$-symbols of the form $(n_r, m_1,m_2,n_r, m_3,m_4)$ are real-valued. This means
\begin{align*}
     \prod_{j=1}^l \left| \begin{array}{ccc}
    n_r & m_1^j & m_2^j \\
    n_r & m_3^j & m_4^j
    \end{array} \right|^2
\end{align*}
is non-negative and implies that the sign of the summand of Equation (\ref{sumof6jspecific}) is determined by the quantum $6j$-symbols associated to the $S-$pieces
\begin{align}\label{monly6j}
\prod_{i = 1}^k \left| \begin{array}{ccc}
    n_r & m^i & m^i \\
    n_r & m^i & m^i
    \end{array} \right|.
\end{align}
By Lemma \ref{Allm6jsignlem}, the sign of each factor in Expression (\ref{monly6j}) is independent of the surface-colorings $m^i$, for $i\in \{1,\dots, k\}$, of the $S-$pieces. 

Therefore, the sign of $|(P,gl)|_{\left(n_r\right)}^{\eta}$ is independent of the surface-coloring $\eta$. From this, we can conclude
\begin{align}\label{newPglformula}
   \left||(P,gl)|_{\left(n_r\right)}\right| &= \left|\sum_{\eta \in adm\left(P,gl,n_r\right)}
   |(P,gl)|_{\left(n_r\right)}^{\eta}\right| \nonumber\\
   &= \sum_{\eta \in adm\left(P,gl,n_r\right)}
   \left||(P,gl)|_{\left(n_r\right)}^{\eta}\right|. 
\end{align}
In fact, since the number of $S-$pieces $k$ is even, every summand of Equation (\ref{sumof6jspecific}) is non-negative, though Equation (\ref{newPglformula}) is sufficient for our purposes.

We can bound the sum in Equation (\ref{newPglformula}) below by the absolute value of a single state $|(P,gl)|_{\left(n_r\right)}^{\eta}$. In particular, we consider the bound obtained from the surface-coloring $\eta = \left(n_r\right)$. This gives us the inequality
\begin{align*}
\liminf_{r \rightarrow \infty} \frac{4\pi}{r} \log \left||(P,gl)|_{\left(n_r\right)}\right| & \geq 
\lim_{r \rightarrow \infty} \frac{4\pi}{r} \log 
\left|\left| \begin{array}{ccc}
    n_r & n_r & n_r \\
    n_r & n_r & n_r
    \end{array} \right|^{k+2l}\right| \\
    &= 2(k+2l)v_8,
\end{align*}
where the equality is due to Equation (\ref{modifiedsharp6j}) in Lemma \ref{modifiedq6jbound}. This means the growth rate of the shadow state sum invariant is bounded below by the simplicial volume $v_3\|M_L(k,l)\| = 2(k+2l)v_8$, establishing the lemma.
\end{proof}

We can now prove Theorem \ref{mainthmrestate}.
\begin{proof}[Proof of Theorem \ref{mainthmrestate}]
Fix a root of unity $q = e^{\frac{2\pi \sqrt{-1}}{r}}$. Let $M = \Sigma_{g}\times S^1$ be a trivial $S^1$-bundle over an orientable closed surface $\Sigma_g$, and let $L\subset M$ be a $2k+2l$ component link such that $M_L(k,l) \in \mathcal{M}$. Suppose $L$ is colored by $\gamma \in I_r^{2k+2l}$, and consider the $I_r$-colored shadow $(P,gl,\gamma)$ associated to $M_L(k,l)$. 

We begin by formulating $TV_r(M_L(k,l);q)$ in terms of the shadow state sum invariant.
By Proposition \ref{TVRTProp} and Theorem \ref{RelRTshadowstatesum}, the Turaev--Viro invariant of $M_L(k,l)$ is given by
\begin{align*}
TV_r(M_L(k,l
);q) &= \sum_{\gamma \in I_r^{2k+2l}} \left|RT_r(M,L,\gamma)\right|^2 \\
	&= \sum_{\gamma \in I_r^{2k+2l}} \left|C_r |(P,gl)|_{\gamma}\right|^2.
\end{align*}

We start with the upper bound, which is proven analogously to the upper bound in Lemma \ref{techlem}. Let $B_r' := (\# I_r)^{2k+2l}$ be the number of link colorings. Both $|C_r|$ and $B_r'$ grow at most polynomially with $r$, so we obtain the following bound. See Theorem 3.3 of \cite{costantinoColoredJones} where $C_r$ is written as a product of terms which grow at most polynomially. 
\begin{align*}
\limsup_{r \rightarrow \infty} \frac{2\pi}{r} \log \left|TV_r(M_L(k,l);q)\right| 
 &= \limsup_{r \rightarrow \infty} \frac{2\pi}{r} \log \left|\sum_{\gamma \in I_r^{2k+2l}}
 \left| |(P,gl)|_{\gamma}\right|^2\right| \\
 &\leq 2 \limsup_{r \rightarrow \infty} \max_{\gamma \in I_r^{2k+2l}} \frac{2\pi}{r} \log \left||(P,gl)|_{\gamma}\right|.
\end{align*}
Since $|(P,gl)|_{\gamma}$ is a product of $k+2l$ quantum $6j$-symbols, we obtain the following bound using Theorem \ref{growthrate6jThm}. 
\begin{align*}
\limsup_{r \rightarrow \infty} \frac{2\pi}{r} \log \left|TV_r(M_L(k,l);q)\right| 
 &\leq 2(k+2l)v_8.
\end{align*}

We now focus on the lower bound. Since all summands are positive, we can bound the absolute value of the sum below by the absolute value of an individual summand. In particular, we consider the bound obtained from the summand corresponding to the $I_r$-coloring $\gamma = \left(n_r \right)$. This gives us the following inequality:
\begin{align*}
\liminf_{r \rightarrow \infty} \frac{2\pi}{r} \log \left|TV_r(M_L(k,l);q)\right| & \geq
\liminf_{r \rightarrow \infty} \frac{2\pi}{r} \log \left|\left|C_r|(P,gl)|_{\left(n_r\right)}\right|^2\right| \\
&= \lim_{r \rightarrow \infty} \frac{4\pi}{r} \log \left|\left|(P,gl)|_{\left(n_r\right)}\right|\right|,
\end{align*}
where equality holds since the limit exists, by Lemma \ref{techlem}, which is equal to the limit inferior and because $|C_r|$ grows at most polynomially. Applying Lemma \ref{techlem}, we obtain the lower bound
\begin{align*}
\liminf_{r \rightarrow \infty} \frac{2\pi}{r} \log \left|TV_r(M_L(k,l);q)\right| & \geq 2(k+2l)v_8.
\end{align*}
Therefore, we can conclude that 
\begin{align*}
\lim_{r \rightarrow \infty} \frac{2\pi}{r} \log \left|TV_r(M_L(k,l);q)\right| &= 2(k+2l)v_8 = v_3 ||M_L(k,l)||.
\end{align*}
\end{proof}

%%%%%%%%%%%%%%%%%%%%%%%%%%%%%%%%%%
%% Section 5
%%%%%%%%%%%%%%%%%%%%%%%%%%%%%%%%%%

\section{Future Directions}\label{FurtherDirections}
By Theorem \ref{growthtrunctetra}, we remark that the quantum $6j$-symbols have a connection with the hyperbolic volumes of truncated tetrahedra \cite{growth6j, Costantino6j2007}. In particular, we use that when all of the sequences of colors grow as $\frac{r}{2}$, the quantum $6j$-symbols have asymptotics corresponding to the volume of an ideal hyperbolic octahedron. One can similarly consider different sequences of colors for the invariants of this family of manifolds. In these cases, the manifolds will not have a complete hyperbolic metric; however, by a conjecture of Wong and Yang \cite{WongYang}, it is expected that the asymptotics of the relative Reshetikhin--Turaev invariants for these sequences of colors still recover geometric properties of the truncated tetrahedra used in the construction. This conjecture has been studied by Wong and Yang for manifolds with complements homeomorphic to either the fundamental shadow links \cite{WongYang} or the figure-eight knot \cite{WongYang2}.   
In our future work, we will investigate this conjecture for the family of manifolds $\mathcal{M}$ in further depth.

%%%%%%%%%%%%%%%%%%%%%%%%%%%%%%%%%%%%%%
\bibliographystyle{abbrv}

\bibliography{biblio}

\begin{thebibliography}{10}

\bibitem{AgolSmall}
I.~Agol.
\newblock Small 3-manifolds of large genus.
\newblock {\em Geom. Dedicata}, 102:53--64, 2003.

\bibitem{BaoBonahon}
X.~Bao and F.~Bonahon.
\newblock Hyperideal polyhedra in hyperbolic 3-space.
\newblock {\em Bull. Soc. Math. France}, 130(3):457–491, 2002.

\bibitem{growth6j}
G.~Belletti, R.~Detcherry, E.~Kalfagianni, and T.~Yang.
\newblock Growth of quantum 6j-symbols and applications to the volume
  conjecture.
\newblock {\em Journal of {D}ifferential {G}eometry}, to appear, 07 2018.

\bibitem{BHMVKauffman}
C.~Blanchet, N.~Habegger, G.~Masbaum, and P.~Vogel.
\newblock Topological quantum field theories derived from the {K}auffman
  bracket.
\newblock {\em Topology}, 34(4):883--927, 1995.

\bibitem{chen2018volume}
Q.~Chen and T.~Yang.
\newblock Volume conjectures for the {R}eshetikhin-{T}uraev and the
  {T}uraev-{V}iro invariants.
\newblock {\em Quantum Topol.}, 9(3):419--460, 2018.

\bibitem{Costantino6j2007}
F.~Costantino.
\newblock {$6j$}-symbols, hyperbolic structures and the volume conjecture.
\newblock {\em Geom. Topol.}, 11:1831--1854, 2007.

\bibitem{costantinoColoredJones}
F.~Costantino.
\newblock Coloured {J}ones invariants of links and the volume conjecture.
\newblock {\em J. Lond. Math. Soc. (2)}, 76(1):1--15, 2007.

\bibitem{CostantinoThurston}
F.~Costantino and D.~Thurston.
\newblock 3-manifolds efficiently bound 4-manifolds.
\newblock {\em J. Topol.}, 1(3):703--745, 2008.

\bibitem{DetcherryCabling}
R.~Detcherry.
\newblock Growth of {T}uraev-{V}iro invariants and cabling.
\newblock {\em J. Knot Theory Ramifications}, 28(14):1950041, 8, 2019.

\bibitem{detcherry2019gromov}
R.~Detcherry and E.~Kalfagianni.
\newblock Gromov norm and {T}uraev-{V}iro invariants of 3-manifolds.
\newblock {\em Ann. Sci. \'{E}c. Norm. Sup\'{e}r. (4)}, 53(6):1363--1391, 2020.

\bibitem{colJvolDKY}
R.~Detcherry, E.~Kalfagianni, and T.~Yang.
\newblock Turaev-{V}iro invariants, colored {J}ones polynomials, and volume.
\newblock {\em Quantum Topol.}, 9(4):775--813, 2018.

\bibitem{Gromov}
M.~Gromov.
\newblock Volume and bounded cohomology.
\newblock {\em Inst. Hautes \'{E}tudes Sci. Publ. Math.}, 56:5--99 (1983),
  1982.

\bibitem{pantsHLS}
A.~Hatcher, P.~Lochak, and L.~Schneps.
\newblock On the {T}eichm\"{u}ller tower of mapping class groups.
\newblock {\em J. Reine Angew. Math.}, 521:1--24, 2000.

\bibitem{HATCHERThurstonPants}
A.~Hatcher and W.~Thurston.
\newblock A presentation for the mapping class group of a closed orientable
  surface.
\newblock {\em Topology}, 19(3):221--237, 1980.

\bibitem{JacoShalen}
W.~H. Jaco and P.~B. Shalen.
\newblock Seifert fibered spaces in {$3$}-manifolds.
\newblock {\em Mem. Amer. Math. Soc.}, 21(220):viii+192, 1979.

\bibitem{Johannson}
K.~Johannson.
\newblock {\em Homotopy equivalences of {$3$}-manifolds with boundaries},
  volume 761 of {\em Lecture Notes in Mathematics}.
\newblock Springer, Berlin, 1979.

\bibitem{KashaevVolConj}
R.~M. Kashaev.
\newblock The hyperbolic volume of knots from the quantum dilogarithm.
\newblock {\em Lett. Math. Phys.}, 39(3):269--275, 1997.

\bibitem{KirillovReshetikhin6j}
A.~N. Kirillov and N.~Y. Reshetikhin.
\newblock Representations of the algebra {${U}_q({\rm sl}(2)),\;q$}-orthogonal
  polynomials and invariants of links.
\newblock In {\em Infinite-dimensional {L}ie algebras and groups
  ({L}uminy-{M}arseille, 1988)}, volume~7 of {\em Adv. Ser. Math. Phys.}, pages
  285--339. World Sci. Publ., Teaneck, NJ, 1989.

\bibitem{LickorishSurgery}
W.~B.~R. Lickorish.
\newblock A representation of orientable combinatorial 3-manifolds.
\newblock {\em Annals of Mathematics}, 76(3):531--540, 1962.

\bibitem{LickorishSkein}
W.~B.~R. Lickorish.
\newblock The skein method for three-manifold invariants.
\newblock {\em J. Knot Theory Ramifications}, 2(2):171--194, 1993.

\bibitem{LickorishBook}
W.~B.~R. Lickorish.
\newblock {\em An Introduction to Knot Theory}.
\newblock Graduate Texts in Mathematics. Springer, New York, NY, 1 edition,
  1997.

\bibitem{Martelli}
B.~Martelli.
\newblock {\em An {I}ntroduction to {G}eometric {T}opology}.
\newblock CreateSpace Independent Publishing Platform, 2016.

\bibitem{MurakamiMurakamiVolConj}
H.~Murakami and J.~Murakami.
\newblock The colored {J}ones polynomials and the simplicial volume of a knot.
\newblock {\em Acta Math.}, 186(1):85--104, 2001.

\bibitem{PerelmanEntropy}
G.~Perelman.
\newblock The entropy formula for the {R}icci flow and its geometric
  applications.
\newblock {\em arXiv preprint at arXiv:math/0211159}, 2002.

\bibitem{PerelmanFinite}
G.~Perelman.
\newblock Finite extinction time for the solutions to the {R}icci flow on
  certain three-manifolds.
\newblock {\em arXiv preprint at arXiv:math/0307245}, 2003.

\bibitem{PerelmanRicci}
G.~Perelman.
\newblock Ricci flow with surgery on three-manifolds.
\newblock {\em arXiv preprint at arXiv:math/0303109}, 2003.

\bibitem{RolfsenBook}
D.~Rolfsen.
\newblock {\em Knots and Links}.
\newblock American Mathematical Society, 2003.

\bibitem{ThurstonGT3manifolds}
W.~P. Thurston.
\newblock {\em The geometry and topology of three-manifolds}.
\newblock Princeton University Math Department Notes, 1979.

\bibitem{ThurstonGeomConj}
W.~P. Thurston.
\newblock Three-dimensional manifolds, {K}leinian groups and hyperbolic
  geometry.
\newblock {\em Bull. Amer. Math. Soc. (N.S.)}, 6(3):357--381, 1982.

\bibitem{TuraevShadow}
V.~G. Turaev.
\newblock Shadow links and face models of statistical mechanics.
\newblock {\em J. Differential Geom.}, 36(1):35--74, 1992.

\bibitem{TuraevBook}
V.~G. Turaev.
\newblock {\em Quantum invariants of knots and 3-manifolds}, volume~18 of {\em
  De Gruyter Studies in Mathematics}.
\newblock Walter de Gruyter \& Co., Berlin, 1994.

\bibitem{TuraevViro}
V.~G. Turaev and O.~Y. Viro.
\newblock State sum invariants of {$3$}-manifolds and quantum {$6j$}-symbols.
\newblock {\em Topology}, 31(4):865--902, 1992.

\bibitem{wallacesurgery}
A.~H. Wallace.
\newblock Modifications and cobounding manifolds.
\newblock {\em Canadian Journal of Mathematics}, 12:503–528, 1960.

\bibitem{wong2020WHFig8}
K.~H. Wong.
\newblock Volume conjecture, geometric decomposition and deformation of
  hyperbolic structures.
\newblock {\em arXiv preprint at arXiv:1912.11779}, 2020.

\bibitem{WongYang}
K.~H. Wong and T.~Yang.
\newblock Relative {R}eshetikhin-{T}uraev invariants, hyperbolic cone metrics
  and discrete {F}ourier transforms {I}.
\newblock {\em arXiv preprint at arXiv:2008.05045}, 2020.

\bibitem{WongYang2}
K.~H. Wong and T.~Yang.
\newblock Relative {R}eshetikhin-{T}uraev invariants, hyperbolic cone metrics
  and discrete {F}ourier transforms {II}.
\newblock {\em arXiv preprint at arXiv:2009.07046}, 2020.

\end{thebibliography}

\Addresses

\end{document}